\documentclass[12pt]{amsart}

\usepackage{amssymb}
\usepackage[all]{xy}
\usepackage[]{graphicx}

\usepackage{multienum}

\usepackage{hyperref}
\hypersetup{
    colorlinks=true,
    linkcolor=blue,
    filecolor=magenta,      
    urlcolor=cyan,
    pdftitle={Sharelatex Example},
    citecolor=teal,
    }

\usepackage[colorinlistoftodos]{todonotes}

\usepackage[mathlines]{lineno}

\newtheorem{thm}{Theorem}
\newtheorem{prop}[thm]{Proposition}
\newtheorem{lem}[thm]{Lemma}
\newtheorem{cor}[thm]{Corollary}

\newtheorem{rem}[thm]{Remark}

\theoremstyle{definition}

\newcommand{\R}{\mathbb{R}}
\newcommand{\N}{\mathbb{N}}

\newcommand{\A}{\mathbb{A}}
\newcommand{\B}{\mathbb{B}}
\newcommand{\Z}{\mathbb{Z}}

\newcommand{\dt}{\frac{d}{dt}}

\newcommand{\Lv}{\frac{\partial L}{\partial v}}

\newcommand{\m}{\mathcal}
\newcommand{\E}{\textsl{E}}

\DeclareMathOperator{\supp}{supp}

\begin{document}
\title{Positive topological entropy of Tonelli Lagrangian flows}

\author[G. Contreras]{Gonzalo Contreras$^1$}
\address{$^1$Centro de Investigaci\'on en Matem\'aticas, A.C., Jalisco S/N, Col. Valenciana CP: 36023 Guanajuato, Gto, México}
\email{gonzalo@cimat.mx}

\author[J. A. G. Miranda]{Jos\'e Ant\^onio G. Miranda$^2$}
\address{$^2$Universidade Federal de Minas Gerais, Av. Ant\^onio Carlos 6627, 31270-901, Belo Horizonte, MG, Brasil.}
\email{jan@mat.ufmg.br}

\author[L. G. Perona]{Luiz Gustavo Perona$^3$}
\address{$^3$Centro Federal de Educa\c{c}\~ao Tecnol\'ogica de Minas Gerais, Av. Amazonas 7675, 30.510-000, Belo Horizonte, MG, Brasil} 
\email{lgperona@cefetmg.br}


\begin{abstract}
We study the topological entropy of the Lagrangian flow restricted to an energy level $E_{L}^{-1}(c) \subset TM$ for $ c >e_0(L)$. We prove that if the flow of the Tonelli Lagrangian $ L: M \to \R$, on a closed manifold of dimension $ n+1$, has a non-hyperbolic closed orbit or an infinite number of closed orbits with energy $ c>e_0(L)$ and satisfies certain open dense conditions, then there exist a smooth potential $ u: M\to \R $, with $ C^2$-norm arbitrarily small, such that the flow of the perturbed    Lagrangian $ L_u=L-u$ restricted to  $E_{L_u}^{-1}(c)$ has positive topological entropy.
The proof of this result is based on an analog version of the Franks' Lemma for Lagrangian flows and Ma\~n\'e's techniques on dominated splitting. As an application,   we show that if   $\dim (M)=2$   and $c > e_0(L)$, then $ L$ admits a $C^2$-perturbation by a smooth potential $u$,  such that, the perturbed flow $\phi_t^{L_u}\big{|}_{E_{L_u}^{-1}(c)}$ has positive topological entropy.     
\end{abstract}

\maketitle


\section{Introduction}

Let $ M $ be a closed smooth manifold of dimension $ n+1\geq 2$ with  a smooth 
Riemannian metric $ \langle \cdot, \cdot \rangle$.
A Tonelli Lagrangian on $ M$ is a smooth function $L: TM\to \R$ satisfying the two conditions:
\begin{itemize}
\item {\it convexity}:
 for each fiber $ T_x M $, the restriction $
L(x,\cdot):T_xM \to \R $ has positive defined  Hessian, 
\item {\it superlinearity}:  
$\lim_{\| v\|\rightarrow \infty} \frac{L(x,v)}{\|v\|} = \infty ,$
uniformly  in $ x \in M$.
\end{itemize}
The  action of $L$ over an absolutely continuous curve $\gamma :[a,b] \to M$
 is defined by
 \[A_L(\gamma) = \int_a^b L(\gamma(t),\dot \gamma(t))\ dt.\]
The  extremal  curves  for the action are given by  solutions of the
{\it Euler-Lagrange equation} which  in local coordinates can be written as:
\begin{equation}
\label{E-L}
\frac{\partial L}{ \partial x}  - \frac{d}{dt }  \frac{\partial L}{\partial v} =0.
\end{equation} 
The  {\it Lagrangian flow}  $ \phi_t^L:TM \to TM $ is defined by $ \phi_t^L( x,v) = (\gamma(t),\dot{\gamma}(t))$ where $ \gamma:\R \to M $ is the solution of {\it Euler-Lagrange equation}, 
with initial conditions   $ \gamma(0)=x$ and $\dot{\gamma}(0)=v.$
The  {\it energy function} $  E_L : TM \rightarrow \R $ is defined by
\begin{equation}
\label{energy function}
 E_{L}(x,v):=\frac{\partial  L}{\partial v} (x,v)\cdot v -L(x,v).
\end{equation}
The  subsets   $ E_L^{-1}(c) \subset TM $  are called {\it energy levels } and they
 are invariant by the Lagrangian flow of $ L$.   
Note that the compactness of $ M$ and the superlinearity of $L$ imply that any non-empty energy levels are compact.
 Therefore the flow $ \phi_t^L$  is defined for all $ t \in \R$. Write 
\begin{equation}
\label{def. e_0}
e_0(L):= \max_{x \in M} E_L(x,0) = - \min_{x \in M} L(x,0).
\end{equation}
 It follows from  the convexity of $ L $,   that 
\begin{equation}
\label{min of E_L in T_xM}
 \min_{v \in T_xM} E(x,v) = E(x,0) \leq  e_0(L). 
 \end{equation}
Thus,  we have that $ \pi( E_L^{-1}(c)) = M $, for all $ c> e_0(L)$.  Observe that, if $ c > e_0$, then the energy level $ E_L^{-1}(c)$
 does not contain any trivial closed orbits, that is fixed points, of $ \phi_t^L$.  

\smallskip

Let $ \omega_0 $ be the canonical symplectic structure in the
cotangent bundle $ T^*M$. 
The theory of Lagrangian flows is closely related to the theory of the Hamiltonian flows on $ (T^*M, \omega_0)$.  
Given a Tonelli Lagrangian $ L:TM \to \R $, the Lagrangian  flow of $ L $ is
conjugated to a Hamiltonian flow in the symplectic manifold $ (T^*M, \omega_0) $ by the
Legendre transformation $ \mathcal L : TM \rightarrow T^*M $,
given by:
\[ \mathcal L (x,v) = \left( x, \Lv (x,v)\right).\]
The corresponding Hamiltonian  $ H: T^*M \rightarrow \R $ is
\[ H(x,p) = \max_{v \in T_xM} \{ p(v)- L(x,v)\}.\]
Conversely, given a smooth, convex and superlinear Hamiltonian $ H: T^*M \to \R$, the corresponding  Lagrangian is 
\[L(x,v) = \max_{p \in T^*_xM} \{ p(v)- H(x,p)\}.\]

\smallskip

Let  $ \m F^r (M) $ be  the   space of smooth functions  
$ u: M \rightarrow \R$, that are called potentials,  endowed with the $C^r$-topology. 
We recall that a subset $ \mathcal O  \subset  \m F^r (M) $ is called  {\it residual} if it contains a countable intersection of open and dense subsets.

We say that a property is {\it generic} (in the sense of Ma\~n\'e) for Tonelli Lagrangians or Hamiltonians, if for each
 $ L:TM\to \R$,  there exists  a residual subset   
$\m O(L)  \subset  \m F^r (M)  $, such that the property holds  for  all $L_u= L -u $, with  $ u \in  \m O(L)$.
In the Hamiltonian setting,  it corresponds to considering perturbations of the kind $H_u=H+u$. 
 This  generically concept  was introduced by R. Ma\~n\'e in \cite{Mane:1996a} to study
the Aubry-Mather theory (see  \cite{Mather:1991}) for generic Tonelli Lagrangians, and  
after this, 
many works were devoted to studying the generic properties of Tonelli Lagrangians and 
Hamiltonians  in the sense  of Ma\~n\'e; see for instance \cite{Contreras_Iturriaga:1999,Contreras_Paternain:2002a,Bernard_Contreras:2008}. 

Also in this context, E. Oliveira proved  in \cite{Oliveira:2008} a version of the
Kupka-Smale Theorem for  Hamiltonians on surfaces. 
In \cite{RR:2011}, L. Rifford and R. Ruggiero proved a local perturbation theorem of 
the $1$-jet of the Poincar\'e map of a closed orbit for  Hamiltonians on  
 closed manifold of any dimension and, using this perturbation theorem and 
Oliveira's results, they obtained   the Kupka-Smale Theorem for Tonelli 
Hamiltonians and Lagrangians on closed manifolds of any dimension. More precisely, they proved that,  for each $ c\in \R $, there exists a residual set 
$\m{KS} = \m{KS} (c) \subset \m F^r(M) $ such that every  Hamiltonian  
$ H_u=H + u $, with  $ u \in  \m{KS}$, if $H_u^{-1}(c)$ is nonempty, then
\begin{enumerate}
\item $H_u^{-1}(c)$ is a regular energy level,
\item
all closed orbits in $H_u^{-1}(c)$ are non-degenerate, and
\item
all heteroclinic intersections in $H_u^{-1}(c)$ are transverse.
\end{enumerate}
\bigskip

In this paper we study the behavior of the functional $   u \mapsto h_{top}(L-u,c) $  for a prescribed energy level $ c>e_0(L)$, where  $ h_{top}(L-u,c)$  denotes the topological entropy of  $\left. \phi_t^u\right|_{ E_{L-u}^{-1}(c)}$.   
The {\it topological entropy}, is an invariant that, roughly speaking, measures its orbit structure complexity.  The relevant question about topological entropy is whether it is positive or vanishes.  Namely, if $\theta\in E_L^{-1}(c)$ and $T,\delta>0$, we define the  $(\delta,T)-${\it dynamical ball} centered at  $\theta$ as 
\[B(\theta, \delta, T)=\{v\in E_L^{-1}(c)\,:\,d(\phi_t(v),\phi_t(\theta))<\delta\mbox{ for all }t\in[0,T]\},\]
 where $d$ is the distance function in $E_L^{-1}(c)$. 
 Let $N_{\delta}(T)$ be the minimal quantity of $(\delta,T)-${\it dynamical ball} needed to cover $E_L^{-1}(c)$. 
 The topological entropy is the limit  $\delta\to 0$ of the exponential growth rate of  $N_\delta(T)$, that is \[h_{top}(L,c):=\displaystyle\lim_{\delta\to 0}\limsup_{T\to\infty}\frac{1}{T}\log N_{\delta}(T).\] Thus, if $h_{top}(L,c)>0$, some dynamical balls must contract exponentially at least in one direction.
\bigskip

In \cite{Carballo_Miranda:2013},
 C. Carballo and J. A. Miranda proved  a perturbation theorem for the $k$-jets, $k\geq 2$, of the Poincar\'e map of 
a non-trivial  closed orbit of the Hamiltonian flow of a Tonelli Hamiltonian $H: T^*M\to \R$
on a closed manifold $M$. As a consequence, they showed that 
there exists a residual subset $\m R= \m R(c) \subset \m F^r (M) $ ($r \geq 4$)  such that, for every $u\in\m R $,
any closed orbit of $H_u$ in $H_u^{-1}(c)$ is either hyperbolic or weakly monotonous 
quasi-elliptic (see \cite[Corollary 5]{Carballo_Miranda:2013}). 
Recall that a non-trivial  closed orbit $\theta$ in $H^{-1}(c)$ is {\it $q$-elliptic orbit}
if the derivative of its Poincar\'e map $P$ has exactly $2q$ eigenvalues 
of modulus $1$, 
and none of them is a  root of 1, 
and 
it is {\it quasi-elliptic} if it is $q$-elliptic for some $q>0$. 
If $\theta$ is a $q$-elliptic orbit, then the central manifold $W^c$ of the 
Poincar\'e map $P:\Sigma \subset H^{-1}(c) \to \Sigma$ of $\theta$ has dimension $2q$, 
$\omega|_{W^c}$ is non-degenerate and $P|_{W^c}$ is a symplectic map on a
sufficiently small neighborhood of $\theta$.
The concept of weakly monotonous is analogous to the twist condition for surface. It implies that the map $P|_{W^c}$, on a
 neighborhood of $\theta$, is conjugated to a symplectic twist map on $\mathbb{T}^q \times \R^q$. The existence of a quasi-elliptic orbit that satisfies the weakly monotonous property has important dynamical consequences, such as the existence of a non-trivial hyperbolic set and then, infinitely many closed orbits and positive topological entropy on the corresponding energy level.  We refer to \cite[Section 4]{Contreras:2010} for details. 

\smallskip

Then we have:
\begin{prop}
\label{energy with a non-hyperbolic orbit}
 Let $L: TM\to  \R$ be a smooth Tonelli Lagrangian. We suppose that the Lagrangian flow has a non-trivial closed orbit with energy $ c $ and that this orbit is not hyperbolic. 
 Then there is a smooth potential $ u:M \to \R $ of $ C^r$-norm arbitrarily small (with $4 \leq r$ ), such that   $ h_{top}(L-u,c)>0 $.
\end{prop}

\smallskip

So, we need to study the situation where 
all closed orbits of the flow $ \left. \phi_t^{L}\right|_{ E_{L}^{-1}(c)} $ are  hyperbolic   and  this situation persist by   perturbations 
of the kind $ L-u$ of $L$, for u in a  neighborhood $ \m U \subset \m F^r(M)$.
\smallskip

We recall that   a compact
 $\phi^L_t$-invariant subset $\Lambda\subset E_{L}^{-1}(c)$  is a {\it hyperbolic 
set} if the restriction of the tangent bundle of $E_{L}^{-1}(c)$ to $\Lambda$ has a splitting
\[ T_\Lambda E_{L}^{-1}(c)= E^s\oplus \langle X^L\rangle\oplus E^u, \]
where $\langle X^L\rangle$ is the subspace generated by the
vector field $X^L$ of $\phi^L_t$, $E^s$ and $E^u$ are
$d\phi^L_t$ invariant sub-bundles and there are constants 
$C,\,\lambda>0$ such that 
\renewcommand{\theenumi}{\roman{enumi}}
\begin{enumerate}
\item\quad $\left|d\phi^L_t(\xi)\right| \leq C\,e^{-\lambda t}\,\left|\xi\right|$ for all
$t>0$, $\xi\in E^s$;
\item\quad $\left| d\phi^L_{-t}(\xi)\right|\leq C\,e^{-\lambda t}\,\left|\xi\right|$ for all
$t>0$, $\xi\in E^u$.
\end{enumerate}
A hyperbolic set $ \Lambda$ is said to be 
 {\it locally maximal }  
 if  there is a neighborhood $U$ of $\Lambda$ such  that  \[  \Lambda=\bigcap_{t \in\R}\phi^L_t(U).  \]  
A {\it non-trivial  hyperbolic basic set} is a 
locally maximal hyperbolic subset $\Lambda\subset E_{L}^{-1}(c)$ which has a dense orbit and is not a single closed orbit.  
 A way of obtaining positive topological entropy is by showing that the flow has a non-trivial  
hyperbolic basic set  \cite{Bowen:1972,Bowen:73} (the topological properties of hyperbolic sets for diffeomorphisms can be seen in \cite[\S~18] {Katok_Hasselblatt:1995}). 

\smallskip

Let us to denote by  $ \m H(L,c)  $ the subset of potentials  $u:M\to \R$ such that all closed orbits of $ \phi_t^{L-u} $ in $E_{L-u}^{-1}(c) $ are hyperbolic and we denote  by  $ \mbox{\rm int}_{C^2} \m H(L,c) $   the interior of  $ \m H(L,c) $ in the $C^2$-topology.
 For $ u \in  \mbox{\rm int}_{C^2} \m H(L,c) $,  let  $Per(L,u,c) $ be the union of all closer orbits of 
 $ \left.\phi_t^{L-u} \right|_{  E_{L-u}^{-1}(c)}$. By definition, $\overline{Per(L,u,c)} \subset E_{L-u}^{-1}(c) $ is a compact and invariant subset.

\smallskip

\smallskip

The main result of this work is the following theorem: 

\begin{thm}\label{T_maim-u} Let  $\dim(M)=n+1$ and let  $L:TM\rightarrow \R$ be a  Tonelli Lagrangian.
   There exists an open and dense  subset $\m G(n,L,c) \subseteq \m F^2(M)$, with   $\m G(1,L,c) = \m F^2(M)$, such
 that, if $u \in \mbox{\rm int}_{C^2} \m H(L,c)\cap\m G(n,L,c)$ and $ c> e_0(L-u)$,
 then  $\overline{Per(L,u,c)}$ is a hyperbolic set.
\end{thm}


    The proof requires the use of a suitable version of Franks' lemma, which was originally proved for diffeomorphisms in \cite{Franks:1971}, by J. Franks. This result is an important tool to study $C^1$-stable properties of a dynamical system. It was used, for instance, by R. Ma\~n\'e in his proof of the structural stability conjecture \cite{Mane:1982}.



In the  case  of geodesic flows, 
in \cite{Contreras_Paternain:2002} G. Contreras and G. Paternain proved a version of the Franks' Lemma for geodesic flows on closed surface,  by performing  $ C^2$-perturbations of the Riemannian metric. 
 In \cite{Contreras:2010} G. Contreras extends    
the proof for geodesic flows 
on closed manifolds of any dimension, again by performing $C^2$-perturbations of the Riemannian metric. 
Also, in \cite{Lazrag_Rifford_Ruggiero:2016}, A. Lazrag, L. Rifford, and R. Ruggiero proved another version of the Franks' Lemma for geodesic flows on a closed manifold of any dimension by performing conformal $C^2$-perturbations of the Riemannian metric; this proof relies on more general techniques from geometric control theory.
By Maupertuis' principle, a small perturbation of a geodesic flow given by a conformal  perturbation of the metric $ g$ is equivalent
to a small perturbation in the Ma\~n\' e sense of the Lagrangian flow, for the Riemannian Lagrangian  $L(x,v)=1/2 g(v,v)$, restricted to the level of energy equal to 1.


We give a proof of a version of Franks' lemma for Tonelli Lagrangians by performing  $ C^2$-perturbations in the Ma\~n\'e's sense, 
that is,  $ C^2$-perturbations  of $ L $ by adding potentials. Note that this kind of $C^2$-perturbations of $L$ is equivalent to a particular  $C^1$-perturbations of the Lagrangian flow $ \phi_t^L$.  
Our proof follows the ideas in \cite{Contreras:2010}  combined with some others from \cite{RR:2011}.
See the details and statements in section~\ref{franksL}.



\smallskip

In Theorem ~\ref{T_maim-u}, if we ensure that the set $\overline{Per(L,u,c)}$ is not a finite union of closed orbits, we can apply standard arguments in dynamical systems (\cite[\S~6] {Katok_Hasselblatt:1995}) and Smale's spectral decomposition theorem for hyperbolic sets to prove the existence of a non-trivial hyperbolic basic set contained in $\overline{Per(L,u,c)}$. This leads us to the following corollary.

\begin{cor}
\label{cor_manyperiodic}
Let  $u \in \mbox{\rm int}_{C^2} \m H(L,c)\cap\m G(n,L,c)$ and $ c> e_0(L-u)$. We suppose that the number of closed orbits of the flow 
$ \left.\phi_t^{L-u}\right|_{E_{L-c}^{-1}(c)} $ is infinite.
Then    $\overline{Per(L,u,c)}$ has a non-trivial hyperbolic basic set.
In particular $ h_{top}(L-u,c) >0$.

\end{cor}

\bigskip

 Let us now  recall the definition of the  Ma\~n\'e's critical value
(cf. \cite{Contreras_Delgado_Iturriaga:1997,Mane:1996,Paternain_Paternain:1997}). 
The {\it critical value}  of $ L $ is the real number $ c(L) $ such that  
\begin{equation}
\label{c(L)}
 c(L)=\inf\{k\in\mathbb{R}:\,A_{L+k}(\gamma)\geq 0, \ \forall \mbox{  closed curve } 
 \gamma:[0,T]\to  M \}.
 \end{equation} 
Let $ p: \tilde{M}\rightarrow M $ be a covering of $ M $ and $ \tilde{L}:T\tilde{M}
\rightarrow \R $ be the lift  of  $ L $ to $ \tilde{M} $, i.e $ \tilde L = L \circ dp$. Then, 
we obtain the critical value $ c( \tilde L ) $ for $ \tilde L $. 
We define the {\it universal critical value} $ c_u(L)$   as the critical value for the 
lift of $ L $ to universal covering  $ p:\tilde{M}_u \to M$. By  (\ref{def. e_0}) and 
(\ref{c(L)}),  we have   that
\[e_0(L) \leq c_u(L)  \leq c(L).\] 

In \cite[Theorem 27]{Contreras_Iturriaga_Paternain_Paternain:2000},  G. Contreras, 
R. Iturriaga, G. P. Paternain, and  M. Paternain showed  that for any  $c>c_u(L)$  
there exists a closed orbit with energy $c$     in any non-zero homotopy class.  Then,
the existence of infinitely many closed orbits with energy  $c>c_u(L)$ can be deduced 
from further properties of the fundamental group of the closed manifold $ M$. 

It is well known that if $ c > c(L) $,  the restriction of the  Lagrangian flow in the   
energy level $ E_L^{-1}(c) $ is a reparametrization of a geodesic flow in the unit
tangent bundle for an appropriate   Finsler metric on $ M$  (cf. \cite[Corollary2]
{Contreras_Iturriaga_Paternain_Paternain:1998a}). Recall that a  {\it  Finsler metric} 
is a continuous function $ F: TM\rightarrow \R $ that is smooth outside the zero section, 
its second derivative in the direction of the fibers is positive defined, and $ F(x,\lambda\ v ) = \lambda F( x,v ) $ for all $ \lambda> 0 $ and $ (x,v)\in TM $.  
If $g $ is a  Riemannian metric  on $ M$, then $ F(x,v) =
g_x(v,v)^{1/2} $ is a trivial  example of a  Finsler metric. We
say that  a Riemannian or  Finsler metric is {\it bumpy } if  all closed
geodesics are {\it non-degenerate}, this is, if the linearized
Poincar\'e map of every closed geodesic does not admit a root of
unity as an eigenvalue.

In the setting of bumpy geodesic flows, 
Rademacher proved in \cite[theorem 3.1(b)]{Rademacher:1989} that if the Riemanniam manifold is  closed, simply-connected manifold, and its rational cohomology ring $
H^*(M,\mathbb{Q})$ is monogenic, then either there are infinitely many closed geodesics,  or there is at least one non-hyperbolic closed geodesic.  
The Rademacher's theorem remains valid for bumpy Finsler metrics on $ S^2$ (cf. \cite[ Section 4 ]{Rademacher:1989}). 
Other stronger results of Franks \cite{Franks:92} and Bangert \cite{Bangert:93} assert that any Riemannian metric on $ S^2$ has infinitely many geometrically distinct closed geodesics.
Many results for geodesic flows
of a Riemannian metric  remain valid  for Finsler metrics but, in
contrast with the  Riemannian case,  there exist examples of bumpy
Finsler metrics on $ S^2$ with only two closed geodesics. These
examples were given by  Katok in \cite{Katok:1973} and were studied
geometrically by Ziller in \cite{Ziller:1983}.

In the low dimensional  case, this is, for a Tonelli Lagrangian $L: TM \to \R$ on a closed surface  ($\dim(M)=2$),  in
 \cite{Asselle_Mazzucchelli:2019}, 
L. Asselle and M.Mazzucchelli showed that for Lebesgue-almost every $c\in 
 (e_0(L),c_u(L))$, the Lagrangian  flow $
\phi_t|_{E_L^{-1}(c)}$ has  infinitely many closed orbits with energy $c$.

\smallskip

So, we combine the above results and obtain:

\begin{thm}\label{thm_surface}
Let $M$ be a closed surface. Let   $L:TM\rightarrow \R$ be a  Tonelli Lagrangian and  $c> e_0(L)$. Then, there is  a smooth potential $u: M \to \R$, that is arbitrarily  $C^2$-small, such that $ h_{top}(L-u,c)>0.$

\end{thm}

We will give the details of the proof of this theorem in the section~\ref{proof surface case}.

\section{Frank's Lemma for Tonelli Lagrangian}\label{franksL}

Let $M$ be a closed smooth manifold of dimension $n+1\geq 2$ and   
 $L:TM\rightarrow\mathbb{R}$ be a Tonelli Lagrangian.
 
Given   an orbit segment 
  $ \phi^L_{[0,t_0]}(\theta) $ contained in some energy level $  E^{-1}(c) $,
 such that,  the arc $ \gamma:=\pi \circ \phi^L(\theta):[0,t_0] \ \to M $ does not contain  self-intersections points.  
  Let $\Sigma_0\,,\Sigma_{t_0} \subset E_L^{-1}(c)$ be  local transverse sections of $ \left.\phi^L_t\right|_{E_L^{-1}(c)}$ at the points $\theta $ and $\phi_{t_0}^L( \theta)$ 
   respectively. Then, if $\Sigma_0 $ is small enough, there is a differentiable function  $ \tau: \Sigma_0 \to \R $ such that  the Poincar\'e map \[P_{t_0}= P_{t_0}(L,\theta, \Sigma_0,\Sigma_{t_0}):\Sigma_0\rightarrow  \Sigma_{t_0} \] defined by \[ P_{t_0}( \nu ) = \phi_{\tau(\nu)} ( \nu ) \in \Sigma_{t_0},\] 
 is a diffeomorphism.  It is a basic fact, that the pullback of the canonical symplectic form $ \omega_0$ on $ T^*M $ by the Legendre transformation $ \m L: TM \to T^* M$ induces a symplectic structure on the transverse sections $ \Sigma_0 $ and $ \Sigma_{t_0}$, and the Poincar\'e map preserves these structures. 
Therefore, via Darboux coordinates, we  identify   the linearized Poncar\'e map \[  d_\theta P_{t_0} : T_\theta \Sigma_0 \to T_{\phi_{t_0}(\theta)} \Sigma_{t_0} \] with a matrix in the symplectic group
 $
  Sp(n)=\left\{ A \in \R^{2n\times 2n} : A^*  
 \left[\begin{smallmatrix}
  0& I \\ 
  -I & 0 
 \end{smallmatrix}\right]
 A=
 \left[\begin{smallmatrix}
  0& I \\ 
  -I & 0 
 \end{smallmatrix} \right] 
 \right\}.
 $

\smallskip

 
Let $ \m F^r (M) $ be the space of all smooth potentials  $u:M\to \R $ endowed  with the 
$ C^r$-topology, $ r \geq 2$.
 Chosen     a tubular 
neighborhood $ W \subset M $ of $\gamma([0,t_0])$,  we set $ \m F^r(\gamma, t_0 ,W) \subset \m F^r(M)$ as the subset of the  potentials  with support in $ W $  that satisfy:  
\[ u(\gamma(t) )=0 \,\, \mbox{ and } \, \, \ d_{\gamma(t)}u =0 \ , \ \forall t \in [0,t_0].
\]
Note that, for any $u\in\m F^r(\gamma,t_0, W)$, the arc $ \phi_t^L(\theta)= (\gamma(t),\dot \gamma(t))$ for $ t \in [0,t_0] $ is also an orbit segment of the Lagrangian flow for the perturbed Lagrangian \[ L_u(x,v):= L(x,v) - u(x) \] and the energy level $E_{L_u}^{-1}(c) $ is tangent to $ E_{L}^{-1}(c) $ along  $ [0,t_0]\mapsto\phi_t^L(\theta)$. 
Then we can consider the map: 
\begin{eqnarray}
\label{def of the map S}
S_{t_0,\theta}:  \m F^r(\gamma, t_0, W) &\longrightarrow & Sp(n)   \\
 u &\longmapsto& S_{t_0,\theta}(u)=d_\theta P_{t_0,u} \nonumber
 \end{eqnarray}
where $d_\theta P_{t_0,u}:T_\theta \Sigma_0 \rightarrow T_{\phi_t(\theta)} \Sigma_{t_0} $ 
is the linearized Poincar\'e map for the  flow of the perturbed 
Lagrangian $ L_u $ in the energy level $ E_{L_u}^{-1}(c) $.

 \smallskip
 

L Rifford and R. Ruggiero in \cite{RR:2011} 
 proved 
that  for all symplectic matrix $ A \in Sp(n) $ sufficiently  closed to  $ S_{t_0, \theta}(0)=d_\theta P_{t_0} \in Sp(n)$
there is  a smooth potential $u:M\to \R $, 
such that,  $ S_{t_0,\theta}(u)= A $. 
Moreover, using methods from geometric control theory, they showed  that given $ \epsilon > 0 $ there is  $ \delta=\delta(\epsilon) >0$, such that:
\[ B(\delta, d_\theta P_{t_0})=\{ A \in  Sp(n) : \|A-d_\theta P_{t_0} \| < \delta \} \subset S_{t_0,\theta} \left( \{ u \in \m F^r(\gamma, t_0, W): \|u\|_{C^r} < \epsilon \} \right), \]
for $ r \geq 2$, see  Proposition 4.2 and Lemma 4.3 in \cite{RR:2011}.

Thus, to prove our  version of the Franks' Lemma for Tonelli Lagrangians (see Theorem~\ref{Franks lemma}), 
we will verify that for the case  
\[S_{t_0,\theta}:  \m F^2(\gamma, t_0, W) \rightarrow  Sp(n)  \ (  \mbox{for r=2}),\]
given $ \epsilon >0$, there is a  radius $ \delta$ that does not depend on $ \theta \in   E^{-1}(c)$ and   $ t_0$ in a fixed closed interval away from zero, such that 
 \[ B(\delta, d_\theta P_{t_0}) \subset S_{t_0,\theta} \left( \{ u \in \m F^2(\gamma, t_0, W): \|u\|_{C^2} < \epsilon \} \right), \]

\smallskip

Before we enunciate our version of the Franks' Lemma, we need to discuss some preliminary results. 

\subsection{Finsler geometry and local coordinates}\

We recall that the Lagrangian flow of $ L $ is conjugated to a Hamiltonian flow on $ ( T^* M, \omega_0) $ by the Legendre transformation  $ \m L( x,v)=( x, \frac{\partial L}{\partial v}(x,v)) $. The corresponding Hamiltonian   $H=H(L):T^*M\rightarrow\mathbb{R}$ is given by :
\begin{equation}
\label{Hamiltonian of L}
H(x,p) = \max_{v \in T_xM} \{ p(v)- L(x,v)\}
\end{equation} 
and  we have the Frechel inequality 
\begin{equation}
\label{Frechel}
p(v) \leq H(x,p) + L(x,v)
\end{equation}
with equality, if only if, $ (x,p)=\m L( x,v)
 $ or equivalently $ p=\frac{\partial L}{\partial v}(x,v) \in T_x^*M$.   Therefore  \[ H\left(x, \frac{\partial L}{\partial v}(x,v) \right) = E (x,v).\]  
 Given a energy level $ E_L^{-1}(c) $ for some  $ c > e_0(L) $, the  set $ H^{-1}(c) := \m L \left(  E_L^{-1}(c) \right) \subset T^* M $ is called Hamiltonian level. 
  
\smallskip

 It follows from the convexity of $ L $   that
 \[ \min_{v \in T_xM} E(x,v) = E(x,0) \leq 
  e_0(L) \]
So,
the zero sections of $TM$ is always contained in the sublevel $
E_L^{-1}( -\infty,  c)  \subset TM.$  
Changing to the Hamiltonian setting, by (\ref{Hamiltonian of L}), we have 
 \[ H(x,0)= - \min_{ v \in T_x M} L(x,v)
 . \]  
 Thus, if $ e_0(L) < c  \leq  - \min_{ (x,v) \in T M} L(x,v) $ the  sublevel $
 H^{-1} (-\infty,c) \subset T^* M $  does not contain the zero section of the cotangent bundle. 


However, it is well know that, if $c>c(L)$, then 
the sublevel $ H^{-1}(-\infty, c)$ contain an exact (lagrangian) graph (cf. \cite{Contreras_Iturriaga_Paternain_Paternain:1998a}),  that is,  there exists  a smooth function  $f:M \to \R $, such that 
\begin{equation}
\label{sub solutions HJ}
H(x,d_xf)<c \ ,\  \  \forall \, x \in M.  
\end{equation} 
Then,  the translated Hamiltonian \[ H_{df}(x,p):=H(x,df+p),\] is such that  the corresponding sublevel $ H_{df}^{-1} (-\infty,c) \subset T^* M $  contains the zero section of the cotangent bundle. It implies that $ H_{df}^{-1}(c) \subset T^*M $ is equal to the unity sphere cotangent bundle of a Finsler metric on $ M$, therefore the flow $ \left.\phi^{H_{df}}_t \right|_{H_{df}^{-1}(c)} $ is a reparametrization of a Finsler geodesic flow.
Moreover,  the Lagrangian flow and the energy functions of  $ L_{df}(x,v):= L(x,v)-d_x f(v)  $ (that is the Legendre transformation of $ H_{df}(x,p) $) coincide with the Lagrangian flow    and energy of  $ 
L $. Thus the Hamiltonian flows $ \phi^{H_{df}}_t $ and $ \phi^{H}_t $ are 
conjugated by the smooth map $ (x,p) \mapsto (x,p+d_xf_0) $. 

If $c \leq c(L)$, the 
inequality (\ref{sub solutions HJ}) does not admit  a solution $f:M\rightarrow R$ (see Theorem A in \cite{Contreras_Iturriaga_Paternain_Paternain:1998a}).  However, the following lemma implies that,  for any energy value $ c \in (e_0, c(L)]$,  there are always local solutions. 

\begin{lem}{\cite[Lemma 2.2]{Asselle_Mazzucchelli:2019}}
\label{solution_local_Asselle}
If $c>e_0(L)$, every point of the closed manifold $M$ admits an open neighborhood $ N \subset M$ and a Tonelli Lagrangian $\tilde{L}:TM\rightarrow\mathbb{R}$ such that $\tilde{L}|_{TN}=L|_{TN}$ and $c(\tilde{L})<c$.
\end{lem}

Therefore, up to a translation by an exact section,  the flow on $ H^{-1}(c)$  (with $ c> e_0(L) $) is locally conjugated to a Finsler geodesic flow. Then, using the exponential map of  Finsler manifolds, the following lemma holds.

\begin{lem}\label{r_injetivity}{\rm (\cite[Lemma 2.3(i)]{Asselle_Mazzucchelli:2019})}.
Let $c>e_0(L)$. There exist $\tau_{inj}>0$ and 
$\rho_{inj}>0$, such that, every point $ x \in M$ admits an open neighborhood $U_{x}\subset M$ containing the compact Riemannian ball $\overline{B_g(x,\rho_{inj})}=\{y\in M; \mbox{dist}(x,y)\leq\rho_{inj}\}$ such that the smooth map \[\psi_{x}:[0,\tau_{inj})\times (E_L^{-1}(c)\cap T_{x}M)\rightarrow U_{x}\] given by $\psi(t,v)=\pi\circ\phi^L_{t}(x,v)$ restricts to a diffeomorphism \[\psi_{x}:(0,\tau_{inj})\times (E_L^{-1}(c)\cap T_{x}M)\rightarrow U_{x}\setminus\{x\}.\]
\end{lem}

\smallskip
As a  corollary of the Lemma~\ref{r_injetivity}, we obtain: 

\begin{cor}
\label{k_0} Let $c>e_0(L)$.
If $k_0:= k_0(L,c)=(\tau_{inj}/4)$,  
 then
\[ \gamma= \pi \circ \phi^L_t(\theta) :[0,2k_0]\rightarrow  M \] 
is  injective,  for  all  $ \theta \in E_L^{-1}(c)$. 
\end{cor}

  \bigskip


We fix  $ c > e_0(L)$ and $ \theta \in E_L^{-1}(c)$.  
 If  $\m L ( \theta)=(x,p) \in H^{-1}(c) $, then $ \phi^H_t(x,p)= ( \gamma(t),p(t))$, where $p(t)=  \frac{\partial L}{\partial v}(\gamma(t),\dot \gamma(t))$. 
Choosing  an open tubular neighborhood  $ W \subset M$ of $\gamma([0,2k_0])$ there is  local coordinates ${x}=(x_0, x_1\cdots, x_n) : W \to \R^{n+1}$, such that 
 \[ x_0(\gamma(t))=t \,\,\mbox{ and }\,\, x_i(\gamma(t))=  0,\] for all $t\in [0,2k_0]$ and $ 1\leq i \leq n$. Then,
if  $p\in T^*_x M$,  with $ x \in W$, we  define $(p_0,p_1,\cdots,p_n)$ by  $p=\sum p_i dx_i$ and we have a natural chart    
\[(x_0, x_1,\cdots, x_n, p_0, p_1, \cdots, p_n): \pi^{-1}(W) \subset T^*M\to \R^{2n+2} \]  that is called {\it natural coordinates } associate to $ x:W\rightarrow \R^{n+1}$.
In natural coordinates the canonical symplectic structure is written as $\omega = \sum_{i=0}^n dx_i \wedge dp_i$ and the Hamilton's equations are:
\begin{equation}
\label{Hamilton's equations}
 \left\{
\begin{array}{ccr}
\dot x_i(t) &=& H_{p_i}(x(t),p(t)) \\
\dot p_i(t) &=& -H_{x_i}(x(t),p(t))
\end{array}
\right.\  \ ,\ \ \ \  
i= 0,1,\dots,n.
\end{equation} 
 Using the basic identity \[\dt  d_\theta \phi_t^H = d_{\phi_t^H(\theta)} X^H  \ d_\theta \phi_t^H ,\]
from (\ref{Hamilton's equations}), we obtain the equations for the linearized Hamiltonian flow, along of  
$ \phi_t^H(x,p)=( \gamma(t), p(t)) $, which we call {\it Jacobi equation}:
\begin{equation}
\label{general Jacobi equation}
\dt
\begin{bmatrix}
a \\ b
\end{bmatrix}
=
\begin{bmatrix}
H_{px} & H_{pp}\\
-H_{xx} & -H_{xp}
\end{bmatrix}
\begin{bmatrix}
a \\ b
\end{bmatrix}.
\end{equation}  

\smallskip

To simplify these equations, 
  we will change the natural coordinates to a special coordinate associated with particular local coordinates in the tubular neighborhood $ W\subset M $ of the segment $ \gamma([0,2k_0])$,   that generalize the well-known Fermi coordinates along a neighborhood of a minimal geodesic.
First, we  observe that 
\[ H\left( \gamma(t), \ \frac{\partial L}{\partial v}(\gamma(t),0)\right) = E( \gamma(t),0) < E( \gamma(t),\dot \gamma(t))=c ,\] 
for all $ t \in \R$.  Thus, taking an exact  1-form $ df_0 $, such that \[ df_0(\gamma(t)) =   \frac{\partial L}{\partial v}(\gamma(t),0),\ \ \ \forall  t \in [0,2k_0] ,\] and, decreasing the neighborhood W if necessary, 
we obtain  that
\[ H(x,d_x f_0) < c,  \ \ \ \ \ \forall \ \ x \in W \subset M.\] 
Therefore, as in Lemma \ref{solution_local_Asselle},
it  implies that the   simple segment $ \gamma: [0, 2k_0] \to M$   admits an open  tubular neighborhood $ W \subset M$ and a Tonelli Lagrangian $\tilde{L}:TM\rightarrow\mathbb{R}$ such that $\tilde{L}|_{TW}=L|_{TW}$ and $c(\tilde{L})<c$. Consequently, the Lagrangian flow restricted  to 
$ TW \cap E_L^{-1}(c)$ is conjugated to a Finsler geodesic flow.
 Then, by an application of a lemma obtained by Li and Nirenberg \cite[Lemma 3.1]{Li_Nirenberg:2005} for Finsler geodesic flows,   we obtain the special local coordinates given in the following result. These coordinates are used also in \cite[ Lemma 3.1]{RR:2011}, and a poof can be seen in \cite[Lemma C1]{Figalli_Rifford:2015}.

\begin{lem}
\label{coordinates}
There exist  a  open neighborhood $ W $ of the segment  $\gamma([0,2k_0])\ \subset M$ and   smooth local coordinates  $x: W \to 
[0,2k_0] \times [-\epsilon_0,\epsilon_0]^n ,
$  with $ x= ( x_0,x_1,\cdots,x_n )$,   such that, 
the Hamiltonian in the natural coordinates  $ (x,p) $ in $  \pi^{-1}(W) 
\subset T^*M $, satisfies  following properties  for every $ t \in [0,2k_0]$:
\begin{itemize}
\item[(a)] $\gamma(t)) = ( t,0,\dots , 0) \ \ \mbox{ and }\ \ \  p(t)=(1,0,\dots,  0)$ 
\item [(b)] $ H_{px}( \gamma(t), p(t)) =H_{xp}( \gamma(t), p(t)) =0 $
\item [(c)]$ H_{pp} ( \gamma(t), p(t)) = \begin{bmatrix}
H_{p_0p_0} &  0 \\
0& I_n 
\end{bmatrix}.$
\item[(d)] $ H_{xx} ( \gamma(t), p(t)) = \begin{bmatrix}
0&  0 \\
0& K(t) 
\end{bmatrix},$ where $ K(t) \in \R^{n\times n} $ denotes the symmetric matrix given by
$ [K(t)]_{ij}=
H_{x_i x_j}(\gamma(t),p(t))$, for ${1\leq i,j\leq n} $, that we call  curvature by 
similarity with the Jacobi equation of the geodesic flows in Fermi coordinates. 
\end{itemize}  
\end{lem}

\bigskip

We fix the tubular neighborhood $W$ of $\gamma([0,2k_0]) $ and the coordinates given by  Lemma~\ref{coordinates}.
For each $t\in[0,2k_0]$, we choose  the local transverse sections $ \Sigma_t \subset H^{-1}(c) $ at $\theta(t):=(\gamma(t), p(t))$, such that $T_{\theta(t)} \Sigma_t =\m N_t $, where \begin{equation}\label{invariant_subspace}
 \m N_t= \{ (a,b)\in T_{\theta(t)} T^*M : a_0=0 ; b_0=0\} \approx \R^{n} \times \R^{n}. 
\end{equation}
It follows from  Lemma~\ref{coordinates}  and  (\ref{general Jacobi equation})  that  the $2n$-dimensional subspaces $ \m N_t \subset T_{\theta(t)} T^*M$ 
is invariant by the differential of the Hamiltonian flow of $ H$ along the orbit $ \theta(t)$. Then, 
 we obtain that
\[ X_0(t):=d_\theta P_t = d_\theta  \left. \phi_t^H \right|_{\m N_0 } : \m N_0 \longrightarrow \m N_t \]
is a solution of the equations:
\begin{equation}
\label{Jacobi equation}
\dot X_0(t)= \begin{bmatrix}
0& I_n \\
-K(t) & 0
\end{bmatrix}
X_0(t) .
\end{equation}

\subsection{General perturbations} \ 

We recall that $ \m F^2(M)$ denotes the space of all smooth potentials  $u: M \to \R $,  with the $ C^2$-topology. 
 Let $\m S(n)$ be  the set of symmetric $ n\times n$  matrices and let  $ \tilde{\m  F}^2(\gamma, k_0 ,W) \subset \m F^2 (M) $ be the subset of  $u:M\to \R$ supported   in $ W$, such that, in the local coordinates of Lemma~\ref{coordinates} $x=(x_0, x_1,\dots,x_n)$, have the form   
\begin{equation}
\label{u(t,x) local}
 u(x)=u(x_0,x_1\dots, x_n ) = \alpha(x_1,\dots,x_n) \sum_{i,j=1}^n \beta 
(x_0)_{ij} \ x_ix_j 
\end{equation}
where $\supp(\alpha) \subset(-\epsilon_0,\epsilon_0)^n $
 and $ \beta:[0,2k_0] \to \m S(n)$ is  supported in $(0,2k_0)$. 
Note that if $ u\in  \tilde{\m  F}^2(\gamma, k_0 ,W)$, then:
 \[u(t,0,\cdots,0)=0 , \ \  d_{(t,0,\cdots,0)}u=0 \]
 and
\[\frac{ \partial^2}{\partial x_0 \partial x_j} u(t,0\cdots,0)= 0 \] 
for all $ t \in [0,2k_0]$ and $ 1\leq j\leq n$. 
Then, for any $ u\in  \tilde{\m  F}^2(\gamma, k_0, W)$, the  Hamiltonian $ H_u:=H+u $ also satisfies  (a), (b), (c) and (d)  in Lemma~\ref{coordinates}.

Therefore, we can consider the map: 
\begin{eqnarray*}
S_t:  \tilde{\m F}^2(\gamma, k_0, W) &\longrightarrow & Sp(n) \\
 u &\longmapsto& X(u,t)
 \end{eqnarray*}
where \[X(u,t):= d_\theta  \left. \phi_t^{H_u} \right|_{\m N_0 } : \m N_0 \longrightarrow \m N_t .\] 

By lemma ~\ref{coordinates} and (\ref{Jacobi equation}), we obtain that  $ X(u,t) $ is a solution of the differential  equation:
\begin{equation}
\label{equation for the map S }
  \dot X(t) = \begin{bmatrix}
0& I_n \\
-K^u(t) & 0
\end{bmatrix} X(t),
\end{equation}
where\begin{equation}
\label{curvature of H^u}
 K^u(t)_{ij}:= K(t)_{ij} + \frac{ \partial^2}{\partial x_i \partial x_j} u(t,0\dots,0)
\ , \ \ \mbox{ for }  \ 1\leq i,j\leq n .
 \end{equation}
Note that, if $ n= \dim(M)-1 =1 $, then $ K^u(t) $ is just a scalar function.

\subsection{Generic Condition } \ 
\label{generic condition}

Let $ \m S(n) $ be the set of $ n \times n $ symmetric matrix. 
Let  \[h_n:\m S (n)\rightarrow [0,\infty) \] be  defined as  
\[h_1(A)= 1  \] 
for all $  A \in  \m S(1)\cong \R$ ($n=1$),
 and  
\[h_n(A)= \prod_{ 0\leq i<j \leq n}( \lambda_i -\lambda_j)^2 , \]
 where $ \lambda_1, \dots, \lambda_n$ are the real eigenvalues of the  $ n \times n $ symmetric matrix $ A$, for any    $n \geq 2$. Observe that $h_n$ is continuous.
We denote
\begin{equation}
\label{eq. function H_n(l,c)}
\Phi_n (L,c)=\Phi_n(H,c) := \min_{\theta \in  E_L^{-1}(c)}\left\{ \max_{t \in \left[0,\frac{k_0}{2}\right]} \  h_n( K(\theta(t) ))\right\}.
\end{equation}
 
 It follows from formula (\ref{u(t,x) local}) and  (\ref{curvature of H^u}), that  the symmetric matrix $K^u(\theta)$ depends continuously on $u\in \tilde{\m F}^2(\gamma, k_0, W)  $ and any smooth path $\alpha(t)$ on $\m S(n)$, 
 with $\alpha(0)=K^u(\theta(0))$ can be realized by the $C^2$ perturbation  $ H+u$ of $ H$,  which preserves the orbit segment $\theta([0,2k_0])$.

Let $ \m G(n,L,c) \subseteq \m F ^2(M) $ be the set of potentials $ u: M \to \R $ defined by
\begin{equation}
\label{eq. set G(n,L,c)}
 \m G(n, L,c)=\m G(n,H,c):=\{u\in\m F^2(M)\,;\,\Phi_n (H+u,c)>0\}.
\end{equation}
Then  $ \m G(1,L,c) = \m F ^2(M) $.

The following proposition was proved in \cite{Contreras:2010} considering the perturbation of the eigenvalues of a one-parameter family of symmetric matrix defined by the same way as (\ref{curvature of H^u}). 

\begin{prop}{\cite[Theorem 6.1]{Contreras:2010}}
\label{generic conditions} Let $ \dim (M)=n+1$.
The set $\m G(n,L,c)$ is open in  $ \m F^2(M)$ and  $\m G(n,L,c)$ is dense in $C^\infty (M)$. 
\end{prop}

\subsection{The Franks' Lemma for Tonelli Lagrangian}\ 

We want to apply a version of the Franks' Lemma to an arc of a closed orbit.  
So, given a closed orbit  $ \phi^L_t(\theta)= (\gamma(t),\dot \gamma(t))$ in $ E_L^{-1}(c)$ the extremal closed curve $   \gamma:\R \to M $  may have points of self intersection. In a more general way, given any finite set of non-self intersecting extremal segments $\mathfrak F=\{\eta_1,\cdots,\eta_m\}$, defined on $[0,2k_0]$, such that the following properties are satisfied:
\begin{itemize}
\item The endpoints of $\eta_i$ are not contained in $W$;
\item The segment $\gamma([0,2k_0])$ intersects each $\eta_i$ transversely.
\end{itemize}
Denote by $\m F^2(\gamma,k_0,W,\mathfrak{F} )$ the set of potentials $u\in\m F^2(\gamma,k_0,W)$ such that $u =0$ in a small neighborhood of $W\cap\cup_{i=1}^m\eta_i([0,2k_0])$.  Thus, for each  closed orbit  $ \phi^L_t(\theta)=(\gamma(t),\dot \gamma(t)) $ in the energy  level $E_L^{-1}(c) $, as in (\ref{def of the map S}), we define the map: 
\begin{eqnarray*}
S_{t_0,\theta}:  \m F^2(\gamma, k_0, W, \mathfrak{F}) &\longrightarrow & Sp(n)   \\
 u &\longmapsto& S_{t_0,\theta}(u)=d_\theta P_{t_0,u} \nonumber
 \end{eqnarray*}
where $d_\theta P_{t_0,u}:T_\theta \Sigma_0 \rightarrow T_{\phi_{t_0}(\theta)} \Sigma_{t_0} $ 
is the linearized Poincar\'e map for the  flow of the perturbed 
Lagrangian $ L_u $ in the energy level $ E_{L_u}^{-1}(c) $,  and $ t_0 \in ( 0,2k_0]$.

\smallskip

We are now ready to state an analog version of the infinitesimal part of Frank's Lemma for Tonelli Lagrangians or Hamiltonians by performing $ C^2$-perturbations in the sense of Ma\~n\'e.

\begin{thm}
\label{Franks lemma} 
Let $L: TM\rightarrow\R $ be a Tonelli Lagrangian on a closed manifold $ M$ of dimension $ n+1 \geq 2$.
 Let $c>e_0(L)$ and suppose that $\Phi_n(L,c)>0$. Then,  given a $ C^2$-open neighborhood $ \mathcal{U} \subset \m F^2 (M)$  of $u_0(x)\equiv 0$; for any  
\begin{itemize}
\item $ k_0\leq t_0\leq 2k_0$;
\item a segment of a extremal curve  $\gamma:[0,2k_0]\rightarrow M$ such that the orbit $( \gamma(t), \dot  \gamma(t)) $ is contained in $  E^{-1}_L(c)$;
\item a tubular neighborhood $ W $ of  $\gamma([0,2k_0])\subset M$;
\item a finite set of non-self intersecting extremal segments $\mathfrak F=\{\eta_1,\cdots,\eta_m\}$, defined on $[0,2k_0]$ such that the segment $ \gamma([0,2k_0])$ intersects each $\eta_i$ transversely,
\end{itemize}
 there exists $ \delta = \delta( L,c,\m U )>0 $ such that the image of the  set $ \mathcal U \cap \mathcal F^2(\gamma,k_0,W,\mathfrak{F})$ under the map 
 $S_{t_0,\theta}$ contains a ball of radius $\delta$ centered at $S_{t_0,\theta}(u_0)=d_\theta P_{t_0}(u_0) \in Sp(n)$.

\end{thm}

To prove this theorem, following the strategy used by G. Contreras in \cite{Contreras:2010}, we will get uniform estimates of the integral equations for the directional derivatives of the map  $S_{t_0,\theta}$ applying to an appropriated finite-dimensional family of potentials, then we will obtain a uniform $ \delta $, for all $ \theta \in E_L^{-1}(c)$ and $ t_0 \in [k_0, 2k_0]$.

\subsection{Proof of the Theorem \ref{Franks lemma}}\ 
 
Given  a Tonelli Lagrangian $ L:TM \to \R $  and  $c>e_0(L)$, we consider the corresponding Hamiltonian $H:T^*M\to \R $, given by (\ref{Hamiltonian of L}) and the Hamiltonian level $ H^{-1}(c)= \m L(E^{-1}(c))$.  
 We assume that  $\Phi_n(L,c)>0$ (see \eqref{eq. function H_n(l,c)}).

Let  $ \m U \subset \m F^2(M) $ be a neighborhood of $ u_0(x)\equiv 0 $, small enough such that $ \m U \subset \m G(n,H,c)$. We will denote $H^u= H+u$.

First, we going to fix some constants and functions that will be useful in the following proposition. 

\begin{itemize}
\item By shrinking $ \m U $ if necessary, let $ k_1=k_1(\m U,c)>0$ be such that the curvature matrix, given in (\ref{curvature of H^u}), satisfies
\begin{equation}
\label{k_1}
\|K^u(\theta(t))\| < k_1,  \ \mbox{ for all } \ \ u \in \m U, \ \  \theta \in 
(H^u)^{-1}(c) \ \  \mbox{ and all } t\in [0,2k_0].
\end{equation}

\item Let $ k_2=k_2(\m U,c) >1 $ be such that if $ u \in \m U $, then 
\begin{equation}
\label{k_2}
\|d_\theta \phi_t^{H^u}\| < k_2  \ \mbox{ and } \ \ \|d_\theta \phi_{-t}^{H^u}\| < 
k_2 , \ \   \mbox{ for all }  \  \theta \in 
(H^u)^{-1}(c) \ \  \mbox{ and all } t\in [0,2k_0].
\end{equation}
\item Given $0 < \lambda << k_0/8$, let $k_3= k_3(\lambda)=k_3(\m U,c,\lambda)$ 
be such that $ \lim_{\lambda \to 0} k_3(\lambda)= 0 $ and
 \begin{equation}
\label{k_3}
\|d_\theta \phi_s^{H^u}- d_\theta \phi_t^{H^u}\| < k_3,  \ \mbox{ and } \ \ \|
d_\theta \phi_{-s}^{H^u}- d_\theta \phi_{-t}^{H^u}\| < k_3 
\end{equation}
for all $ |s-t|< \lambda$,  $s,t\in [0,2k_0]$, $ u \in \m U $ and all  $\theta \in 
(H^u)^{-1}(c)$. 
\item Choose $ \lambda= \lambda(\m U, c) $ small enough  such that $0< 
k_3(\lambda) <1/{2 k_2^3} $, then
 \begin{equation}
\label{lambda}
  \frac{1}{k_2^2} - 2 k_2 k_3(\lambda) >0.
\end{equation}

\item 
 Since the Hamiltonian  $ H$  satisfies  the condition $  \Phi_n(H,c)>0$, there exists $ a_0 >0 $ such that $ \Phi_n(H,c) > 2 a_0^2$. 
 Therefore

\begin{equation}
\label{map h(0,theta)}
   \max_{t \in \left[\frac{k_0}{4},\frac{3k_0}{4}\right]}
  \  h_n( K(\theta(t)))
  \geq 2 a_0^2 
\ \ \mbox{ for all } \ \theta \in H^{-1}(c).
\end{equation}

\item Consider  $ \tilde h_n ( u,\theta, t ) := h_n (K^u(\theta(t)) $. Since that the 
map $ h_n $ is continuous and the $ n\times n $ matrix $ K^u(\theta(t)) $ depends continuously on $ u 
\in \m F^2(M) $,  the map    $u \mapsto \tilde h_n  $ is continuous. Let
\[ A_0:= \{ (\theta,t) \in (H^u)^{-1}(c) \times [0,2k_0] : \tilde h_n ( 0, \theta, t ) \geq 2 
a_0^2\}   .\]
Then $ A_0 \subset (H^u)^{-1}(c) \times [0,2k_0] $ is a compact set, and by (\ref{map 
h(0,theta)}), we have that
$ A_0 \cap \left\{ \{\theta\}\times \left[\frac{k_0}{4},\frac{3k_0}{4}\right] \right\} 
\not=\emptyset ,$
for all $ \theta \in H^{-1}(c)$.   Since $ \tilde h_n $ is continuous, there exists a 
neighborhood  $ \m U_0 \subset \m U $ of $ 0 \in \m F^2(M) $ such that 
$ \tilde h_n(u,\theta,t ) > a_0^2$, for all $(u,\theta, t ) \in \m U_0 \times A_0.$ 

\item If $ n \geq 2 $, we choose a time $ \tau = \tau(\theta, \m U)$ such that $ ( \theta, \tau ) \in A_0$, and
\[ a_0^2 <    \prod_{ 1\leq i<j \leq n} ( \lambda_i -\lambda_j)^2
\leq \left( 2 \|H_{xx}^u(\theta(\tau))\|\right) ^{2(m-1)} |\lambda_i- \lambda_j|^2
 \leq ( 2 k_1) ^{2(m-1)} |\lambda_i- \lambda_j|^2, \]
 for all $ u \in \m U _0 $, where $ m =m(n)= \left(\begin{array}{c} 
n\\ 2 \end{array}\right) = \displaystyle\frac{n(n-1)}{2}$. Therefore  
\begin{equation}
\label{k_4}
\min_{i\not=J} |\lambda_i-\lambda_j|>\frac{a_0}{(2k_1)^{m-1}}= k_4(n):=k_4.
\end{equation}
\item Let 
\begin{equation}
\label{k_5}
k_5:=k_5(n)=\left\{\begin{array}{lc}
\max \{ k_4^{-1}(n),\  1 + 4k_1k_4^{-1}(n),\ 1 + k_1 \}, & \mbox{ if } n\geq 2 ,\\
1 + k_1, & \mbox{ if } n=1
\end{array} \right.
\end{equation}
\item  Let $ \delta:[0,2k_0] \to [0 ,\infty) $ be a $ C^\infty$ function such that 
$ \delta(s) = 0 $ if  $ |s-\tau |\geq \lambda $, and
$ \int_0^{t_0} \ \delta(t)\ dt = 1 $ , where $ \lambda(\m U_0) $ is from 
(\ref{lambda}). The $ C^5$ norm of $\delta $ depends only $ \m U_0 $ and does not 
depend on $ \tau(\theta, \m U_0).$
\item By (\ref{lambda}), there exists $ \rho= \rho(\m U) >0 $ such that
\begin{equation}
\label{rho and k_6} 
k_6:= \frac{k_2^{-2}-2k_2k_3-\rho k_2^2 \|\delta\|_{C^0}}{k_2k_5} >0.
\end{equation} 

\item Given  a finite set of non-self intersecting extremal segments
  $\mathfrak F=\{\eta_1,\ldots,\eta_m\}$, defined on $[0,2k_0]$,
   such that  the segment $\gamma([0,2k_0])$
 intersects each $\eta_i$ transversely.
 Let $h:[0,2k_0]\to[0,1]$ be a $C^\infty$ function with support outside
  the intersecting points
  \[
  \supp(h)\subset (0,2k_0) \setminus
   (\gamma)^{-1}\left[\cup_{i=1}^m\eta_i\right]
  \]
  and such that  
\begin{equation}
\label{f, rho}
\int_0^{t_0} ( 1-h(t) ) dt < \rho\,.
\end{equation}   

\item 
Given $ \epsilon >0 $, let $\alpha_\epsilon: [-\epsilon,\epsilon]^n \to 
\R $ be a  $ C^\infty$-function   
  such that $ \alpha_\epsilon(x_1,\dots,x_n)=1 $  if  $(x_1,\dots,x_n) \in 
[-\frac{\epsilon}{4},\frac{\epsilon}{4}]^n $, and  $ \alpha_\epsilon(x_1,\dots,x_n)=0 
$  if  $(x_1,\dots,x_n) \notin 
[-\frac{\epsilon}{2},\frac{\epsilon}{2}]^n $.
\item Finally , we fix  $ \epsilon_1=\epsilon_1(\gamma, \mathfrak{F},W)  
<\epsilon_0 $ such that in the local coordinates $ x: W \to [0,2k_0]\times 
[-\epsilon_0, \epsilon_0]^n $, given by Lemma~\ref{coordinates},  the segments $ 
\eta_i  $ do not intersect the set $ \supp(h) \times [-\epsilon_1, \epsilon_1]^n$.   

\end{itemize}

\smallskip

Once that $\m S(n)$ are symmetric $n\times n$ matrices and denoting by ``$*$'' the transposed operator of the matrix, we can write:
\begin{eqnarray*}
  \m S(n) &=& \{ a \in \R^{n\times n} : a^*=a\},\\
 \m S^*(n) &=& \left\{ d\in  \m S(n) : d_{11}= \dots =d_{nn}=0 \right\} ,\\
  \m{AS}(n) &=&\{ e  \in \R^{n\times n} : e^*=-e\} .
\end{eqnarray*}
Note that if $ M $ is a surface, that is $ n=1$, then $ \m S^*(1)=\m{AS}(1) =\{ 0\}$.
 
Now we are ready to start the following proposition:


\begin{prop}
\label{prop. cota derivada} Let $ t_0\in [k_0,2k_0]$.
Let $ F: \m S(n)^3 \times \m S^*(n) \to Sp(n) $ be the map defined as
\[ \m S(n)^3 \times \m S^*(n) \ni  (a,b,c,d)=w  \longmapsto F(w)
=\left. d_\theta \phi_{t_0}^{\left(H^{u(w)} \right)} \right|_{\m N_0} ,\] 
where 
\[ H^{u(w)}(x,p)= H(x,p)+ u(w)(x)\]
with
\begin{equation}
\label{definitions of u(w)(t,x)}
 u(w)(t,x_1,\dots x_n)= \alpha_{\epsilon}(x_1,\dots x_n)  \ \sum_{i,j=1}^n [\beta(w)
 (t)]_{ij} \ x_i x_j ,
 \end{equation}
for some $ 0<\epsilon\leq\epsilon_1$, and 
\begin{equation}
\label{definitions of beta(w)(t)}
\beta(w)(t)=h(t) \left[ a \delta(t) + b\delta^\prime (t) + c 
\delta^{\prime\prime}(t) + d \delta^{\prime\prime\prime}(t) \right] .
\end{equation}
Then, if $ u(w) \in \m U_0 $ , there exist $k_6=k_6(H,c,\mathcal{U})>0$ such that
\[ \left\| d_w F ( \xi) \right\| \geq k_6 \|\xi\| \,,\]
for all $ \xi  \in  \m S(n)^3 \times \m S^*(n)\approx \R^{2n^2+n} $.
 \end{prop}

\smallskip

\begin{proof}
Given $\xi=(a,b,c;d) \in \m S(n)^3 \times \m S^*(n)$,    consider the curve 
 \[\R\ni r \longmapsto \Gamma(r):=w + r\xi  \in \m S(n)^3 \times \m S^*(n).\] 
Then \[ d_wF(\xi) =\left. \frac{ d }{dr}F(\Gamma(r)) \right|_{r=0}
 =  \left. \frac{d}{dr} \left.\left( d_\theta \phi_{t_0}^{H^{u(\Gamma(r))}} \right|_{\m N_0}\right)\right|_{r=0} .\]
 We denote \[ X_r(t) =  d_\theta \phi_{t}^{H^{u(\Gamma(r))}}:\m N_0 \to \m N_t \] 
 for $ t\in (0,2k_0)$. It follows from   (\ref{equation for the map S }) for $H + u(\Gamma(r))$ that
\begin{equation}\label{rJac}
\dot X_r(t)=\begin{bmatrix} 0 & I \\ -(K(t)+r \beta(t)) & 0
\end{bmatrix}  X_r(t),
\end{equation}
where $\beta(t)=\beta(\xi)(t)$ is from~(\ref{definitions of beta(w)(t)}). 

Let $Z(t)=\left.\frac{d}{dr} X_r(t)\right|_{r=0}$. Deriving  the equation (\ref{rJac}) with respect to 
$r$, we get the differential equation 
\begin{equation}
\label{dZ}
\dot Z = \A Z +\B  X,
\end{equation}
where $\A=\begin{bmatrix} 0 & I \\ -K(t) & 0
\end{bmatrix}$ and
$\B=\begin{bmatrix} 0 & 0 \\ \beta(t) & 0
\end{bmatrix}$. 



\smallskip

Write $Z(t) = X(t) Y(t)$, then from~(\ref{rJac}) and~(\ref{dZ}) we get that 
\[ X \,\dot Y = \B\, X.\]
Since $X_r(0)\equiv I$, we have that $Z(0)=0$ and $Y(0)=0$.
Therefore
\[
Y(t)=\int_0^t X^{-1}(s)\,\B(s)\, X(s)\; ds.
\]

Note that 
\begin{eqnarray}
\frac{\B(s)}{h(s)}= 
\left[\begin{array}{cc}
0 & 0 \\
\frac{\beta(s)}{h(s)} & 0 \end{array} \right]
&=&\left[\begin{array}{cc}
0 & 0 \\
a\delta(s)+b\delta'(s)+c\delta''(s)+d\delta'''(s) & 0 \end{array} \right]\nonumber\\
\nonumber \\ 
&=&\delta(s)A+\delta'(s)B+\delta''(s)C+\delta'''(s)D \label{eq_aintegrar}
\end{eqnarray}
where 
\[
  A=\begin{bmatrix} 0 & 0 \\ a & 0 \end{bmatrix},
  \quad
  B=\begin{bmatrix} 0 & 0 \\ b & 0 \end{bmatrix},
  \quad
  C=\begin{bmatrix} 0 & 0 \\ c & 0 \end{bmatrix},
  \quad
  D=\begin{bmatrix} 0 & 0 \\ d & 0 \end{bmatrix}
  \]

\smallskip 

Integrating by parts the last three  terms of (\ref{eq_aintegrar}) and using~(\ref{rJac}), we have that
 \def\Cc{\begin{bmatrix} c & 0 \\ 0 & -c
  \end{bmatrix}}
  \def\Dd{\begin{bmatrix} \phantom{-}0 & -2d
    \\ -(Kd+dK) & \phantom{-}0 \end{bmatrix}}
  \begin{align*}
 \int_0^{t_0} & X^{-1}(s) \,\delta'(s)\,B X(s)\;ds =\\
    &=\int_0^{t_0} \delta(s)\,\left(\, X^{-1}(s)\dot X(s)\,X^{-1}(s) \, B\, X(s)
    - X^{-1}(s) B \,\A(s)\, X(s)\,\right)\; ds=
    \\
    &=\int_0^{t_0} \delta(s) \, X^{-1}(s)  [\A(s) , B ] X(s) \; ds
    =\int_0^{t_0} \delta(s) \, X^{-1}(s)
          \begin{bmatrix} b & 0 \\ 0 & -b
      \end{bmatrix}
      X(s) \; ds.
   \end{align*}
  
  \smallskip
  
   \begin{eqnarray*}
   \int_0^{t_0} X^{-1}(s) \,   \delta''(s) \,   C X(s)\;ds &
   =&\int_0^{t_0} \delta'(s) \, X^{-1}(s) [\A(s),C] X(s) \; ds=
   \\
 & =& 
   \int_0^{t_0} \delta(s) \, X^{-1}(s) \ \left[ \A(s),  [\A(s),C]\right] \ 
    X(s) \ ds=
    \\
    &=&\int_0^{t_0} \delta(s) \, X^{-1}(s)
    \begin{bmatrix} \phantom{-}0 & -2c
    \\ -(K(s)c+cK(s)) & \phantom{-}0 \end{bmatrix}
    X(s) \; ds.
  \end{eqnarray*}
  
  \smallskip
     
  \[\int_0^{t_0} X^{-1}(s) \,\delta'''(s)\,     D X(s)\;ds
 =\int_0^{t_0} \delta'(s) \, X^{-1}(s)  \left[ \A(s),  [\A(s),D]\right]X(s)s \  ds=
    \]    
\begin{eqnarray*}  & =&\int_0^{t_0} \delta(s) \, X^{-1}(s) \
  \left[ \A(s),  [ \A(s),  [\A(s),D] ]\right] X(s) \    ds=
    \\
 \hspace{2cm}   &=&\int_0^{t_0} \delta(s) \, X^{-1}(s)
        \begin{bmatrix} -K(s)d-3dK(s) & 0
    \\ 0 & 3K(s)d+dK(s) \end{bmatrix}
    X(s) \; ds.
 \end{eqnarray*}
 
Then we have that
    \begin{equation}\label{W1}
  W_{t_0}:=\int_0^{t_0} X^{-1}(s)\  \frac{\B(s)}{h(s)}\  X(s)\ ds
  =\int_0^{t_0} \delta(s) \ X^{-1}(s)\ 
  T(\xi) \  X(s)\ ds
  \end{equation}
where $ T: \m S^3(n) \times \m S^*(n) \to \mathfrak{sp}(n)  $ is the linear map
\begin{equation}
\label{map T}
T(\xi )=T(a,b,c,d)= \begin{bmatrix} \beta & \gamma \\
  \alpha & -\beta^*\end{bmatrix}, 
\end{equation}
defined  by 
 \begin{eqnarray}
  \alpha &=& a - (Kc+cK), \notag\\
  \gamma &=& -2c\,,  \label{syst}\\
  \beta &=& b - Kd - 3dK. \notag
 \end{eqnarray}
We want to solve the above equations   for the variables $a,b,c\in S(n)$ and $d\in S^*(n)$, where $\alpha,\gamma\in S(n)$ and $\beta\in\R^{n\times n}$ are arbitrary parameters.

\begin{rem}
\label{remark SL surface case}
In the  case of $ \dim(M)=2$, we have $ d =0\in \R $, therefore the corresponding system {\rm (\ref{syst})} is a $ 3\times3 $ linear system that  has the solution 
  \begin{eqnarray*}
   a &= &\alpha - \frac{1}{2} (K\gamma +\gamma K). \\
  b &= &\beta,  \\
  c &= &-\frac{1}{2}\,\gamma ,
  \end{eqnarray*}
 for $ K=K(\tau) \in \R$, with $ \tau \in [\frac{k_0}{4}, \frac{3k_0}{4}]$ fixed.
 \end{rem}
  
  If $ \dim(M)=n+1 \geq 3$,  we start by separating $\beta$ into a sum of a symmetric matrix and an antisymmetric matrix. Thus we obtain the following auxiliary equation 
  \begin{equation}\label{b-b*}
  Kd - dK = \frac{\beta - \beta^*}2.
   \end{equation}

We recall that since we suppose that $ \Phi_n(H,c)>0 $, we can choose the time $ \tau \in [k_0/4,3k_0/4] $ for which all eigenvalues of   $ K(\tau)$ are distinct and satisfy (\ref{k_4}).  The following lemma gives a matrix $d \in S^*(n)$  of equation (\ref{b-b*}).

\begin{lem}{\cite[Lemma7.4]{Contreras:2010}}
  \label{Lemma of subjectivity}
  For each $ n \geq 2$, let $K$ be a $ n \times n$  symmetric matrix  and let
  $L_K:\m S^*(n)\to \m{AS}(n)$ be given by
  \[L_K(d):=K d-d K.\]
  Suppose that the eigenvalues $\lambda_i$ of $K$ are all distinct.
  For all $e\in \m{AS}(n)$ there exists $d\in\m S^*(n)$
  such that $L_K\,d=e$ and
  \[
  \| d\| \leq \frac {\| e\|}{\min\limits_{ i\ne j}|\lambda_i-\lambda_j|}.
  \]
   \end{lem}
   
The rest of the solution to the system~\eqref{syst}, for $ n \geq 2$, is given by
  \begin{eqnarray}
 a &=& \alpha - \frac{1}{2} (K\gamma +\gamma K).\label{a}\\ 
  b &=&\frac{1}{2}\,(\beta+\beta^*) + 2\, (Kd+dK),\label{b} \\
  c &=& -\frac{1}{2}\,\gamma ,\label{c} 
  \end{eqnarray}
  That extends the solutions in Remark~\ref{remark SL surface case} for surface to manifolds of arbitrary dimension.
 
Since that  
   \[\dim(\m S(n)^3\times\m S^*(n))=\dim(Sp(n))=n(n^2+1)=2n^2+n,\] we prove that $[d_w F]$ is a linear  isomorphism. Moreover, its norm satisfies the inequality given in the following lemma. 

\begin{lem}
\label{quot_const}
There exist a constant $k_6=k_6(H,c,\mathcal{U})>0$ such that
 \[
 \|Z({t_0})\| = \| d_wF(\xi)\| \geq k_6\ \|\xi\|.   \]
 for all $ \xi\ \in \m S(n)^3\times \m S^*(n) \approx \R^{2n^2+n}$.
\end{lem} 

\end{proof}


\begin{proof}

Let $ T: \m S^3(n) \times \m S^*(n) \to \mathfrak{sp}(n)  $ be the linear map given by (\ref{map T}).
First we will estimate  $\| T^{-1}\|$. 
 Recall that if $ \dim(M)=2$, then $ \|d\|=0 $ and by Remark~\ref{remark SL surface case} $ \|b\|=\|\beta\|$.
 For $ n =\dim(M)-1 \geq 2$, observing  that
  \[
  \| \beta\| = \sup_{|u|=|v|=1}\langle \beta\,u,v\rangle
             = \sup_{|u|=|v|=1}\langle u, \beta^* \,v \rangle
         = \|\beta^*\|
  \]
 and from~\eqref{b-b*}, Lemma~\ref{Lemma of subjectivity}, \eqref{k_4} 
  and~\eqref{k_5},
  \begin{equation}\label{Nd}
  \| d\| \leq
  \frac{\big\Vert\frac{\beta-\beta^*}2\big\Vert}
       {\min\limits_{i\ne j}|\lambda_i-\lambda_j|}
  \leq \frac{\|\beta\|}{k_4}
  \leq k_5\,\|\beta\|.
  \end{equation}
   From~(\ref{b}), (\ref{k_1}),  (\ref{Nd}) and~(\ref{k_4}),
  \[
  \| b\|\leq\|\beta\|+ 4\,k_1\,\| d\|
  \leq\left( 1 + 4\,k_1\,k_4^{-1}\right)\,\|\beta\|
  \leq k_5\,\|\beta\|.
  \]
  Also, from (\ref{c}), (\ref{k_4}) and (\ref{a}),
 \[ \| c\|\leq \|\gamma\|\leq k_5\,\|\gamma\|,\]
 \[ \| a\|\leq\|\alpha\|+k_1\,\|\gamma\|
          \leq k_5\,\max\{\|\alpha\|,\|\gamma\| \} \]

  Since
  \[
  \| T(\xi) \|\geq \max\{\|\alpha\|,\|\beta\|,\|\gamma\|\},
  \]
  we get that
  \[
  \|\xi \|:=\max\{\| a\|,\| b\|,\| c\|,\| d\|\}
  \leq k_5\,\| T(\xi)\|.
  \]
  Thus
  \begin{equation}\label{eD}
   \| T(\xi)\| \geq \frac{ 1}{k_5} \| \xi\|.
  \end{equation}

  Write
\[
  Q(s):=X^{-1}(s)\,T(\xi)\,X(s)\qquad \text{ and }\qquad
  P(s):=\delta(s)\;X^{-1}(s)\,T(\xi) \,X(s).
 \]
   Given a continuous map $f:[0,2k_0]\to\R^{2n\times 2n}$, define
   \[
   \m O_\lambda(f,\tau):=\sup_{|s-\tau|\leq\lambda}|f(s)-f(\tau)|.
   \]
   Observe that
   \[
   \m O_\lambda(f g,\tau)
   \leq \| f\|_0\,\m O_\lambda(g,\tau)+\m O_\lambda(f,\tau)\,|g(\tau)|,
   \]
      where $\| f\|_0:=\sup\limits_{s\in[0,2k_0]}\  | f(s)|$.
   We have that
  \begin{eqnarray*}
   \m O_\lambda(Q,\tau) &=&\m O_\lambda(X^{-1}(s)T(\xi) X(s),\tau)
   \leq \\
   & \leq &\| X^{-1}(s)\|_0\,\m O_\lambda(T(\xi) X(s),\tau)
   +\m O_\lambda(X^{-1}(s),\tau)\,\| T(\xi)\|\,\| X(\tau)\| \leq
   \\
   &\leq &\| X^{-1}(s)\|_0\,\|T(\xi)\|\,\m O_\lambda(X(s),\tau)
   +\m O_\lambda(X^{-1}(s),\tau)\,\| T(\xi)\|\,\| X(\tau)\| \leq
   \\
   &\leq & 2 \,k_2k_3\,\|T(\xi)\|.
\end{eqnarray*}   
 
 \[ \| W_{t_0}-Q(\tau)\|
   =\left\|\int_0^{t_0}\delta(s)\ [Q(s)-Q(\tau)]\ ds \right\|
   \leq \m O_\lambda(Q,\tau)
   \leq 2 \,k_2 k_3\,\|T(\xi)\|. \]
   
 \begin{eqnarray*}
   \| Y(t_0)-W_{t_0}\|
   = \left\| \int_0^{t_0}\  [h(s)-1]\ P(s)\ ds\right\|
   &\leq &\| P\|_0 \int_0^{t_0}\  |1-h(s)|\ ds\\
   &\leq & \rho\;\| P\|_0
   \leq \rho\ k_2^2\  \|\delta\|_0 \  \|T(\xi)\|.
 \end{eqnarray*}  

   \[  \| Q(\tau)\|=
   \| X_\tau^{-1}\ T(\xi)\ X_\tau\|
   \geq \frac{1}{k_2^2}\|T(\xi)\|.
   \]
 
   Therefore
   \begin{eqnarray*}
   \| Y(t_0)\|
   &\geq &\| Q(\tau)\|-\| W_{t_0}-Q(\tau)\|-\| Y(t_0)-W_{t_0}\|
   \\
   &\geq &\left(\frac{1}{k_2^2}-2 k_2\ k_3-\rho_0\ k_2^2\|\delta\|_0\right)\|T(\xi)\|.
\end{eqnarray*}  
   Using~(\ref{eD}),   
   \[
 \|Z({t_0})\| = \| X({t_0})\ Y({t_0})\|
   \geq k_2^{-1}\ \| Y({t_0})\|
   \geq \frac{k_2^{-2}-2 k_2\ k_3-\rho\ k_2^2\ \|\delta\|_0}{k_2\ k_5}
   \,\|\xi\|  = k_6\ \|\xi\|.
   \]

\end{proof}

\smallskip

The proof of the following two general lemmas can be seen in \cite{Contreras:2010}.  

 \begin{lem}{\cite[Prop. 7.3]{Contreras:2010}}
 \label{general lemma of open maps}
 Let $ N $ be a smooth connected Riemannian manifold, with $ \dim(N)=m$, and let $ F: R^m \rightarrow N $ be a smooth map such that 
 \begin{equation}
 \label{condition for open maps}
 \| d_x F (v)\| \geq a >0\ \ ,\ \mbox{ for all } \ \ (x,v) \in T\R^m \ \ \mbox{ with } \  
 \|v\|=1 \ \mbox{ and } \ \|x \| \leq r   .
 \end{equation}
 Then for all $ 0< b< ar $,
 \[ \{ w \in N : d ( x, F(0))< b\}  \subset   F ( \{ x \in \R^m : \| x\| < b/a \} ) \]
\end{lem}

\smallskip

\begin{lem}{\cite[Lemma~7.6]{Contreras:2010}}
\label{normC2}
There exists a constant $ k_7>0$ and a family of  $ C^\infty$-function $\alpha_\epsilon: [-\epsilon,\epsilon]^n\rightarrow \R $  such that $ \alpha_\epsilon(v)=1 $  if  $v \in [-\frac{\epsilon}{4},\frac{\epsilon}{4}]^n $, and  $ \alpha_\epsilon(v)=0 $  if  $v\notin [-\frac{\epsilon}{2},\frac{\epsilon}{2}]^n $, and for any $C^2$ map $ B:[0,1] \to \R^{n\times n } $ the functions  \[ G(t,v):= \alpha_\epsilon(v) v^* B(t) v\] satisfies:

\begin{equation}
\label{k7}
\| G\|_{C^2} \leq k_7 \|B\|_{C^0} + \epsilon k_7 \|B\|_{C^1} + \epsilon^2 k_7 \|B\|
_{C^2},
\end{equation}
with $ k_7$ independent of $ 0 < \epsilon <1 $.
\end{lem}

\bigskip

Then, we consider  the map $ G_\epsilon :  \m S(n)^3 \times \m S^*(n) \times \to   \m F^2(\gamma, k_0,W, \mathfrak{F}) $ defined as  
\[ G_\epsilon(w) = \alpha_\epsilon(x_1,\dots x_n)  \ \sum_{i,j=1}^n [\beta(w)(t)]_{ij} 
\ x_i x_j , \]
where $ \beta(w):[0,2k_0] \to \m S(s) $  is from (\ref{definitions of beta(w)(t)}). The following diagram commutes
\[
  \m S(n)^3 \times \m S^*(n) \supset 
  \xymatrix{ {B(0, k_6^{-1} \eta)}\ar[r]^{ G_{\epsilon}\ } 
   \ar[rd]_{F}& \m 
   F^2(\gamma, k_0,W, \mathfrak{F})  \ar[d]^{S_{t_0,\theta} }  \\
& Sp(n)  } 
\]   
for some $ 0<\epsilon \leq \epsilon_1$. 
 It follows from the Proposition~\ref{prop. cota derivada} and  Lemma~\ref{general 
 lemma of open maps},  that
\[
B(S_{t_0}(0),\eta)  \subset   F( B(0, k_6^{-1} \eta))  \subset Sp(n).\]

\bigskip

Therefore, to conclude the proof of the Theorem~\ref{Franks lemma} it remains to prove that  \[G_{\epsilon}(B(0,k_6^{-1}\eta))\subset\m U_0 \subset \m F^2(M).\] 
It follows from
\begin{lem}
\label{computing the norm}
For $\epsilon $ small enough, there exist $ \eta ( \m U_0, c ) $, such that, 
\[
 \|w\|\leq \eta \Rightarrow 
G_{\epsilon}(w) \in \m U_0.\]

\end{lem}  

\begin{proof}
  
Let $ 0< r_0= r_0(\m U_0) < 1 $ be  such that 
\begin{equation}
\label{epsilon-0}
\|u\|_{C^2} < r_0 \Rightarrow u \in \m U_0,
\end{equation}
and let $ \eta = \eta(  \m U_0 )$ be such that 
\begin{equation}
\label{eta}
4 \eta k_6^{-1}k_7 \|\delta\|_{C^3 } < \frac{r_0}{2}.
\end{equation}
and let  $ \epsilon_2 = \epsilon_2(\m U_0, \gamma, W) \leq \epsilon_1 $ be such 
that 
\begin{equation}
\label{epsilon-1}
k_6^{-1} \eta ( 8 k_7\epsilon_2 \|h\|_{C^1}\|\delta\|_{C^4} + 16 k_7 \epsilon_2^2
\|h\|_{C^2}\|\delta\|_{C^5} ) < \frac{r_0}{2}.
\end{equation}

So, if $ \|w\| < k_6^{-1}\eta $,\ $ \beta(w)(t) $ from (\ref{definitions of beta(w)(t)}), and $ \epsilon \leq  \epsilon_2 $, then, by Lemma~\ref{normC2},
\begin{eqnarray*}
\|G_{\epsilon}(w)\|_{C^2} &\leq& k_7 \| \beta(w)\|_{C^0} + \epsilon k_7 \|\beta(w)\|
_{C^1} + \epsilon^2 k_7 \|\beta(w)\|_{C^2} \\
& \leq &  k_7\, 4\, k_6^{-1}\eta \|\delta\|_{C^3 } + k_7 \epsilon_2 \,4\,k_6^{-1} \eta \left(2\|h\|
_{C^1}\|\delta\|_{C^4}\right) + \\
&+&  k_7\epsilon_2^2 \,4\,k_6^{-1} \eta\left(2^2\|h\|_{C^2}\|\delta\|_{C^5}\right) \\
& \leq & \frac{r_0}{2}+ \frac{r_0}{2} 
\end{eqnarray*}
where the last inequality is from (\ref{eta}) and (\ref{epsilon-1}). Then, by 
(\ref{epsilon-0}), $ G_\epsilon( w) \in \m U_0$.

\end{proof}

\section{Lagrangian Flow with infinitely many periodic orbits in an energy level}

In this section, we will prove the Theorem~\ref{T_maim-u}. The proof is based on Ma\~n\'e's techniques about the theory of stable hyperbolicity, developed for the proof of the stability conjecture in \cite{Mane:1982}. 
This technique
requires the use of a suitable version of Franks' lemma.

\smallskip

Let $L: TM\rightarrow\mathbb{R}$ be a Tonelli Lagrangian on a closed manifold $ M$  and $ c > e_0(L)$.
Let  $ \m H(L,c) $ be the set of smooth potentials  $u:M \to \R$ such that all closed orbits of the flow \[ \left.\phi_t^{L-u} \right|_{  E_{L-u}^{-1}(c)} : E_{L-u}^{-1}(c) \rightarrow E_{L-u}^{-1}(c) \] are hyperbolic . We denote by   $ \mbox{\rm int}_{C^2} \m H(L,c) $  the interior of  $ \m H(L,c) $ in the $C^2$-topology.
 For each  $ u \in  \mbox{\rm int}_{C^2} \m H(L,c) $,  let  $Per(L,u,c) \subset  E_{L-u}^{-1}(c) $ be the union of all periodic orbits  of the flow 
 $ \left.\phi_t^{L-u} \right|_{  E_{L-u}^{-1}(c)}$.

 By definition, $\overline{Per(L,u,c)} \subset E_{L-u}^{-1}(c) $ is a compact and invariant subset. We will prove the following theorem:

\begin{thm}\label{t-per-hip}
If $u\in \mbox{\rm int}_{C^2} \m H(L,c) \cap\m G(n,L,c)$, then the set \[\overline{Per(L,u,c)}\] is a hyperbolic set.
\end{thm}

\smallskip
 
Let us to  that the flow $ \left.\phi_t^{L-u} \right|_{E_{L-u}^{-1}(c)} $ has infinitely many    hyperbolic closed orbits.
Standard arguments of dynamical systems \cite{Katok_Hasselblatt:1995} imply that the set $\overline{Per(L,u,c)}$ is locally maximal. Hence, it follows by Smale's Spectral Decomposition Theorem \cite{Smale:1967} (see too \cite{Katok_Hasselblatt:1995}) that $\overline{Per(L,u,c)}$ has a non-trivial basic set. Then, it follows: 

\begin{cor}
\label{c-per-hip} 
Let  $u\in \mbox{\rm int}_{C^2} \m H(L,c) \cap\m G(n,L,c)$ be such that  the flow $ \left.\phi_t^{L-u} \right|_{E_{L-u}^{-1}(c)} $ has infinitely many    hyperbolic closed orbits. Then $\Lambda:=\overline{Per(L,u,c)}$ has a non-trivial basic set. In particular $\left.\phi_t^{L-u}\right|_{E_{L-u}^{-1}(c)}$ has positive topological entropy.
\end{cor}

\bigskip

Let us recall  some definitions  that will be useful  to prove the 
Theorem~\ref{t-per-hip}. 

\smallskip

We say that a linear map $T: \R^{2n} \rightarrow \R^{2n} $ is {\it hyperbolic} if $ T $  has no eigenvalue of norm equal to 1. The {\it stable and unstable subspaces of $T$} are defined as \[ E^s(T) = \left\{v \in \R^{2n};\ \lim_{n \rightarrow \infty }
T^n(v) =0 \right\}\ \mbox{ and } \  E^u(T) = \left\{v \in \R^{2n};\
\lim_{n \rightarrow \infty } T^{-n}(v) =0 \right\},\]
respectively.

Let $Sp(n)$ be the group of symplectic linear isomorphisms of $\R^{2n}$. We say that a sequence  $ \xi: \Z \rightarrow Sp(n)$ is {\it periodic} if there is $ n_0\geq 1 $ such that $ \xi_{i+n_0}= \xi_i$, for all $ i \in \Z$. We say that a periodic sequence $\xi$ is {\it hyperbolic} if  the linear map $ \prod_{i=0}^{n_0-1}
\xi_i $ is hyperbolic. In this case, we denote the stable and
unstable subspaces of $ \prod_{i=0}^{n_0-1} \xi_{j+i} $  by $
E_j^s(\xi) $ and $ E_j^u(\xi) $, respectively.

We say that a parametrized family $\xi=\{\xi^{\alpha}\}_{\alpha\in\mathcal{A}}$ of sequences in $Sp(n)$ is {\it bounded} if there exist $Q>0$ such that $||\xi^{\alpha}||<Q$ 
  for each $\alpha\in\mathcal{A}$ and $i\in\Z$.

 Given two families  of periodic sequences  $\xi^{\alpha} = \{\xi^{\alpha}: \Z \rightarrow Sp(n);\ \alpha \in \mathcal A \} $
and $ \eta^{\alpha}=\{ \eta^{\alpha}: \Z \rightarrow Sp(n);\
\alpha \in \mathcal A \}$, we say that they are {\it periodically equivalent} if they have the same index  set  $  \mathcal A $ and the minimal period of $ \xi^{\alpha} $ and $\eta^{\alpha} $ coincide, for each $\alpha \in \mathcal A $. Given two periodic equivalent families of periodic sequences in $Sp(n)$, $\xi=\{\xi^{\alpha}\}_{\alpha\in\mathcal{A}}$ and $\eta=\{\eta^{\alpha}\}_{\alpha\in\mathcal{A}}$, we define:
\[ d(\xi,\eta) =
\sup\{ \| \xi_n^{\alpha}- \eta_n^{\alpha}\| ;\
\alpha \in \mathcal A, \ n \in \Z \}.\]

We say that a family $ \xi^{\alpha} $ is a {\it hyperbolic } if for each $\alpha\in\mathcal{A}$, the periodic sequence $\xi^{\alpha}$ is hyperbolic. We say that a hyperbolic periodic family is {\it stably hyperbolic}
if there is $\epsilon > 0$ such that  any family $\eta$ periodically equivalent to $\xi$ satisfying  $ d
(\xi,\eta)< \epsilon $ is also hyperbolic.

Finally, we say that a family of periodic sequences $\xi$ is a {\it uniformly hyperbolic} if 
there exist $K>0,\,0<\lambda<1$ and invariant subspaces $E_i^s(\xi^{\alpha}),\,E_i^u(\xi^{\alpha})$, $\alpha\in\mathcal{A},\,i\in\Z$, such that
\[\xi_j(E^{\tau}_j(\xi^\alpha))=E^{\tau}_{j+1}(\xi^\alpha)\mbox{ for all }\alpha\in\mathcal{A},\,j\in\Z,\,\tau\in\{s,u\}\] and
\[\left|\left|\prod_{i=0}^{m-1}
\xi_{i+j}^{\alpha}\bigg|_{E_j^s(\xi^{\alpha})}\right|\right|
<K\lambda^m\mbox{ and } \left|\left|\left(\prod_{i=0}^{M}
\xi_{i+j}^{\alpha}\bigg|_{E_j^u(\xi^{\alpha})}\right)^{-1}\right|\right|<K\lambda^m,\] for all $\alpha\in\mathcal{A},\,j\in\Z,\,m\in\N.$ Observe that in this case, the sequence $\xi$ is hyperbolic and the subspaces $E^s_i(\xi^{\alpha})$, $\E^u_i(\xi^{\alpha})$ necessarily coincide with the stable and unstable subspaces of the map $\prod_{j=0}^{m-1}\xi_{i+j}^{\alpha}$.

In \cite{Contreras:2010}, G. Gontreras proved that:

\begin{thm}[\cite{Contreras:2010}, Theorem 8.1]\label{Contreras2010_T8.1}
If $\xi^{\alpha}$ is a stably hyperbolic family of periodic sequences of bounded symplectic linear maps then it is uniformly hyperbolic
\end{thm}

\subsection{Proof of theorem \ref{t-per-hip}}\

Let  $c>e_0(L)$ and let  $U\subset M$ be an open set. For any   potential  $ u \in \m F^2(M)$, 
we denote by  $\m P(L,u,c,U)$  the set of all closed orbit $\phi_t^{L-u}(\theta)$ in $E_{L-u}^{-1}(c)$ such that $\gamma(t)=\pi\circ\phi_t^{L-u}(\theta) $ are entirely contained in $U$ and  we define \[Per(L,u,c,U):=\displaystyle\bigcup_{\phi_t^{L-u}(\theta)\ \in \ \m P(L,u,c,U)}\{ \phi_t^{L-u}(\theta) : t \in \R\},\] 
\[\m H(L,c,U):=\{u\in\m F^2(M)\,;\,\,\mbox{all  orbit in } Per(L,u,c,U)\mbox{ is hyperbolic }\}.\]

\begin{prop}
\label{t-per-hip-loc}
If $u\in \mbox{\rm int}_{C^2}\m H(L,c,U) \cap\m G(n,L,c)$, then the set \[\Lambda:=\overline{Per(L,u,c,U)}\] is a hyperbolic set.
\end{prop}

\begin{proof}
Let $c>e_0(L)$ and let $ k_0=\rho_{inj}/4$,be given by Corollary~\ref{k_0}. For each $\alpha\in Per(L,u,c,U)$, there is $\theta= (x,v) \in \pi^{-1}(U)\cap E_{L-u}^{-1}(c) $ such that $\alpha(t)=\phi_t^{L-u}(\theta)=(\gamma(t),\dot{\gamma}(t))$. Let $T_{\alpha}$ be the minimal period of $\alpha$ and $n=n(\alpha,u)$ such that $T_{\alpha}=nt_0$ for some $t_0\in(k_0,2k_0]$. For each $ 0 \leq i \leq n-1$, we define  the
segment $ \gamma_i(t) = \gamma( it_0 + t ) $ and let $ \Sigma_i \in T(\pi^{-1}(U)) $ be the local transversal sections at $ \alpha(it_0)=\phi_{it_0}^{L-u}(\theta) $  such that $ T_{\alpha(it_0)} \Sigma_i =\mathcal N(\alpha(it_0))=\mathcal N (i,\alpha) $. We denote by
\[ \xi_{i,u}^{\alpha}= dP(u, \Sigma_i,\Sigma_{(i+1)}):\mathcal N(i,\alpha) \rightarrow \mathcal N(i+1,\alpha)\] the linearized Poincar\'e maps. 

Consider the family \[\xi(u)=\{\xi^{\alpha}\}_{\alpha\in Per(L,u,c,U)}:=\{\xi_{i,u}^{\alpha}\,;\,\alpha\in Per(L,u,c,U)\}.\]

\bigskip

In the proof of the following lemma, we must use   the Franks' Lemma for Tonelli Lagrangians (see 
Theorem\ref{Franks lemma}).

\begin{lem}
\label{linear h-stability}
The family $\xi(u)=\{\xi^{\alpha}\}_{\alpha\in Per(L,u,c,U)}$ is stably hyperbolic.
\end{lem}

\begin{proof}
Since $u\in\m F^2(L,c,u)$, there is a $C^2$-neighborhood $\m U\subset\m F^2(L,c,u) $ such that each orbit of $Per(L,u,c,U)$ has a hyperbolic analytic continuation $\phi_t^{L-\tilde{u}}(\theta)=\beta(\alpha)$ with $\tilde{u}\in \m U$, because otherwise we could produce a non-hyperbolic orbit. Then each $\beta(\alpha)$ intersects the section $\Sigma_i$, with $0\leq I \leq n-1$ and we can cut $\beta(\alpha)$ into the same number of segments of $\alpha$. So, if we denote by \[\tau_{i,\tilde{u}}^{\beta}=dP(\tilde{u},\Sigma_i,\Sigma_{i+1}):T_{\phi_i^{L-\tilde{u}}(\theta)}\Sigma_i\rightarrow T_{\phi_{i+1}^{L-\tilde{u}}(\theta)}\Sigma_{i+1}\] the linearized Poincar\'e maps associated to $\beta(\alpha)=\phi_t^{L-\tilde{u}}(\theta)$, we have that the family
\begin{equation}\label{fam-per}
\xi(\tilde{u}):=\left\{\tau_{i,\tilde{u}}^{\beta}\,;\,\beta \mbox{ is the analytic continuation of }\alpha \mbox{ with } 0\leq i \leq n-1 \right\},
\end{equation}
is hyperbolic.

Let be $\delta=\delta(L,c,u)>0$ given by the Frank's Lemma (Theorem \ref{Franks lemma}).

We suppose that $\xi(u)$ is not stably hyperbolic. Then there exists a closed orbit $\alpha \in Per(L,u,c,U) $ and a sequence of linear symplectic maps $\eta_i\in Sp(n)$ 
such that:
\[ \| \xi_{i,u}^{\alpha} - \eta_i \| <\delta \,\, \mbox{ and }\,\, \prod_{i=0}^{n(\theta_t,u)} \eta_i \, \mbox{  is not hyperbolic}.\]
Note that the perturbation space in the  Franks' Lemma preserves the selected orbit $\alpha$. It follows from the Fanks' Lemma that there is a potential $\tilde{u_i}\in\m{U}\cap \mathcal F^2(\theta_t,k_0, W,\mathfrak{F})$ such that $\alpha$ is a periodic orbit of $\phi_t^{L-\tilde{u_i}}$ and $\tau_{i,\tilde{u_i}}^{\alpha}=\eta_i $. Let $\tilde{u}:=\tilde{u}_1+\cdots +\tilde{u}_n$. Since that
 \[ d_{\theta_0}P(\tilde{u}, \Sigma_0,\Sigma_0) = \prod_{i=0}^{n(\alpha,u)}
 \tau_{i,\tilde{u}}^{\alpha}=\prod_{i=0}^{n(\alpha,u)} \eta_i,\]
 the orbit $ \alpha $ is a non-hyperbolic orbit for the Euler-lagrange flow $\phi_t^{L-\tilde{u}} $. This contradicts the choice of $\mathcal U $.
\end{proof}

Since that  $\xi(u)$ is a stably hyperbolic family, applying the Theorem \ref{Contreras2010_T8.1}, we have that $\xi(u)$ is uniformly hyperbolic and we can obtain a hyperbolic continuous splitting on $Per(L,u,c, U)$. Then, using the continuity, we extend the hyperbolicity condition to the closed set $\Lambda=\overline{Per(L,u,c, U)}$. Thus $\Lambda$ is a hyperbolic set.
\end{proof}

\section{Lagrangian flow with a non-hyperbolic closed orbit}
\label{section non-hyperbolic}

In this section, we going to prove the Proposition~\ref{energy with a non-hyperbolic orbit}.

\smallskip

 First, we recall some facts about the jet space of symplectic diffeomorphisms.
We consider the space of smooth diffeomorphisms 
$f : \R^{m}\to \R^{m}$ such that $f(0) =0 $.
Given $ k \in \N $, we say that the diffeomorphisms $f$ and $g$ are 
{\em $k$-equivalent} if their Taylor polynomials of degree  $k$ at $0$ are equal.
The {\em $k$-jet} of a diffeomorphism $f$ at $0$, $j^k f (0)$, or $j^k f$ for short,
is the equivalence class of $f$.
If we consider only symplectic diffeomorphisms in $\R^{2n}$, 
the set of all the equivalence classes is the {\em space of symplectic $k$-jets}, 
that we denote $J^k_s(n) $.
Observe that $ J^k_s(n)$ is 
a Lie group, 
with the product defined by
\[j^k f \cdot j^k g = j^k ( f\circ g ).\]
We say that a subset $Q \subset J^k_s(n)$ is {\em invariant} if
\begin{equation}
\label{inv-subset-J}
\sigma \cdot Q \cdot \sigma^{-1} = Q ,  \forall   \sigma \in J^k_s(n) .
\end{equation}

\smallskip

Note that  if $M$ has dimension $n+1$, 
$(\gamma(t),p(t))$ is a non-trivial closed orbit of the Hamiltonian flow of $H:T^*M\to\R$
in  the energy level $H^{-1}(c)$, and
$\Sigma \subset  H^{-1}(c)$ is a local transverse section 
at the point $ (\gamma(0),p(0))$,
then $\Sigma$ also has a symplectic structure and the 
Poincar\'e map $P(\Sigma,H):\Sigma \rightarrow \Sigma $ is a symplectic
diffeomorphism.
Therefore, using Darboux coordinates, we can assume that
$j^k P(\Sigma,H))\in J^k_s(n)$.

\bigskip

Given  an invariant subset 
$Q \subset  J^k_s(n)$ and a  closed orbit $\theta$ of a Hamiltonian flow,  
it follows from (\ref{inv-subset-J}) that the property 
``the  $k$-jet of the  Poincar\'e map of $ \theta$ belongs to $Q $'' 
is  independent  of  the section $\Sigma$ and the coordinate system; 
hence, it is well-defined.

The following local perturbation theorem was proved by C. Carballo and J.A. Miranda in \cite{Carballo_Miranda:2013}.

\begin{thm}
\label{T-jet}
Let $M$ be a closed manifold of dimension $n+1$, $H: T^*M\to\R$,  
$k\geq 1$, and
$Q$ any open invariant subset of $J^k_s(n)$.
Suppose that $
(\gamma(t),p(t))$ is a non-trivial closed orbit 
of the Hamiltonian flow of $H$.
If the $k$-jet of the Poincar\'e map of $(\gamma(t),p(t))$ is in $\overline{Q}$,
then there exists a smooth potential $u: M\to \R $, 
$C^r$-close to zero, $r\geq k+1$, such that
\begin{enumerate}
\item
$(\gamma(t),p(t))$ is also a closed orbit of the Hamiltonian flow of $H+u$  and
\item
the $k$-jet of the Poincar\'e map of $(\gamma(t),p(t))$, 
as a closed orbit of the Hamiltonian flow of $H+u$, is in $Q$.
\end{enumerate}
\end{thm}

Recall that a closed orbit $\theta$ in $H^{-1}(c)$ is {\em $q$-elliptic}
if the derivative of its Poincar\'e map $P$ has exactly $2q$ eigenvalues 
of modulus $1$ which are non-real;
it is {\em quasi-elliptic} if it is $q$-elliptic for some $q>0$. 
If $\theta$ is a $q$-elliptic closed orbit, then the central manifold $W^c$ of the 
Poincar\'e map $P:\Sigma\to \Sigma$ of $\theta$ has dimension $2q$, 
$\omega|_{W^c}$ is non-degenerate, and $F=P|_{W^c}$ is a symplectic map on a
sufficiently small neighborhood of $\theta$.

Let 
$\lambda_1, \dots, \lambda_q, \overline{\lambda_1}, \dots, \overline{\lambda_q}$ 
be the eigenvalues of modulus $1$ of the derivative $d_\theta P$ of the Poincar\'e map of 
a $q$-elliptic orbit $\theta$.
We say that $\theta$ is {\em $4$-elementary} if
$\prod_{i=1}^q \lambda_i^{m_i}\neq 1$ 
whenever the integers $m_i$ satisfy $1\leq\sum_{i=1}^q |m_i|\leq 4$.

Let us recall the Birkhoff's Normal  Form (for a proof  see
 \cite[p.~101]{Klingenberg:1978}).

\begin{thm}[Birkhoff Normal Form]
\label{bnf}
Let $F: \R^{2q}\to \R^{2q}$ be a symplectic diffeomorphism  such that $0$ is a 
$q$-elliptic $4$-elementary fixed point for $F$.  
Then
there are symplectic coordinates $(x,y)= (x_1, \dots, x_q, y_1, \dots, y_q)$ 
in a neighborhood of $0$ such that
\[
F_k(x,y)=F_k(z)=e^{2\pi i\phi_k(z)} z_k+ R_k(z), \quad k=1, \dots, q,
\]
where
$z= x+ iy$, $\phi_k(z)= a_k+ \sum_{l=1}^q \beta_{kl} |z_l|^2$,
$\lambda_k= e^{2\pi i a_k}$,
and $R_k(z)$ has zero derivatives up to order $3$ at $0$. 
\end{thm}

Then, we say that a $q$-elliptic closed orbit $\theta$ is {\em weakly monotonous}
if the coefficients $\beta_{kl}$ of $F=P|_{W^c}$ satisfies 
$\det[\beta_{kl}]\neq 0$. Observe that what we call here weakly monotonous is a generalization of the twist condition
of surface diffeomorphisms. The weak monotonous conditions for a q-elliptic closed curve imply that the restriction of the Poincar\' e map, defined for a local transversal section of this curve in the energy level, is conjugated to a twist map on $ \mathbb{T}^q \times \R^q$. 

\smallskip

We shall use the following result:

\begin{thm}[{\cite[Theorem 4.1]{Contreras:2010}}]
\label{existence of 1-elliptic}
If $ F: \mathbb{T}^q \times \R^q \to  \mathbb{T}^q \times \R^q$ is a $C^4$-Kupka-Smale weakly monotonous  exact symplectic diffeomorphism which is $ C^1$ near a symplectic completely integrable  diffeomorphism $ G$, then $ F$ has a 1-elliptic periodic point near $ \mathbb{T}^q \times \{0\}$. In particular, there is a non-trivial hyperbolic set for $F$ near  $ \mathbb{T}^q \times \{0\}$ and $ h_{top}(F)\not= 0$.
\end{thm}

We are now ready to show Proposition~\ref{energy with a non-hyperbolic orbit}.

\subsection{Proof of Proposition~\ref{energy with a non-hyperbolic orbit}} \

 Observe that, by the Birkhoff Normal Form,  the set  $Q_0\subset J^3_s(n)$ of $C^\infty$ symplectic diffeomorphisms
$F$ on $\R^{2n}$ such that the origin is 
a weakly monotonous $4$-elementary $q$-elliptic fixed point, for some $q>0$
 is open and invariant.
 
 If the Hamiltonian flow $ \phi_t^H=\m L\circ \phi_t^L $   has a non-hyperbolic closed orbit $ (\gamma(t),p(t))$ in the energy level $H^{-1}(c)$, then the $3$-jet of the Poincar\'e map of $(\gamma(t),p(t))$, 
as a closed orbit of the Hamiltonian flow of $H$, is an element of   $\overline{Q_0}$.  By applying the  Theorem~\ref{T-jet},  we obtain a potential  $ u: M\to \R$ of arbitrary small $ C^r$-norm, such that,   $ \left. \phi_t^{H+u} \right|_{ {H+u}^{-1}(c)} $
has a weak monotonous q-elliptic closed orbit.  Also, we can assume that $ \left. \phi_t^{H+u}\right|_{ {H+u}^{-1}(c)} $ satisfies the Kupka-Smale Theorem.    

 The weak monotonous conditions for a q-elliptic closed curve implies that the restriction of the Poincar\' e map, defined for a local transversal section of this curve in the Hamiltonian level, to the $2q$ dimensional central manifold, is conjugated to a Kupka-Smale weakly monotonous ( a twist map) F on $ \mathbb{T}^q \times \R^q$. In \cite{Moser:1977}, Moser proves that there is a subset  $ \mathbb{T}^q \times \mathbb{B}_\delta$ near $ \mathbb{T}^q \times \{0\}$ and a iterate $ N \in \N$, such that, the map $ F^N $ on  $ \mathbb{T}^q \times \mathbb{B}_\delta$ is a weakly monotonous exact symplectc diffeomorphism    which is $ C^1$-near to a symplectic completely integrable  diffeomorphism $ G$. In this case, the Theorem~\ref{existence of 1-elliptic} says that $  F$ has a 1-elliptic periodic point.
 It is proved in  \cite[section 4]{Contreras:2010}, that the restriction of the Poincar\'e map of this 1-elliptic orbit to its two-dimensional central manifold is a twist map on the annulus $ S^1\times \R $, then has a homoclinic orbit (see for example \cite[remarques p34]{Le-Calvez:1991}).  Since the central manifold is normally hyperbolic and the Kupka-Smale conditions imply the whole Poincar\'e map has a homoclinic point, therefore it has positive topological entropy.

\section{Tonelli Lagrangians on surfaces: }
\label{proof surface case}

In this section, we will consider the low dimensional case, i.e. $\dim(M)=2$.

Recall that the universal critical value $ c_u(L)$  is the critical value (\ref{c(L)}) for the lift of $ L $ to the universal covering  $ p:\tilde{M}_u \to M$. 
This is equivalent to saying that
\begin{equation}
\label{c_u(L)}
 c_u(L)=\inf\{k\in\mathbb{R}:\,A_{L+k}(\gamma)\geq 0, \ \forall \mbox{ contractible  closed curve } 
 \gamma:[0,T]\to  M \},
 \end{equation} 

In \cite{Contreras_Iturriaga_Paternain_Paternain:2000},  G. Contreras, 
R. Iturriaga, G. P. Paternain, and  M. Paternain showed  that for any  $c>c_u(L)$  
there exists a closed orbit with energy $c$     in any non-zero homotopy class.  Then,
the existence of infinitely many closed orbits with energy  $c>c_u(L)$ can be deduced 
from further well-known properties of the fundamental group of a closed surface. For example, it implies  that 

\begin{lem}
\label{prop infinite closed orbit for high energy} 
If $ M \not= S^2$  then, for any  $ c>c_u(L) $, the flow $ \left.\phi^L_t \right|_{E_L^{-1}(c)} $ has infinity many closed orbits.
\end{lem} 
\begin{proof}
Since that the surface $ M $  has genus $ g \geq 1 $,  its fundamental group $\pi_1(M)$  that has at least two generators free of torsion.
Then, 
and applying  \cite[Theorem 27]{Contreras_Iturriaga_Paternain_Paternain:2000} to each homotopy class in the subgroup generated by these two generators, we certainly get
 infinitely many closed orbits in $E_L^{-1}(c)$ which are geometrically different.
\end{proof}

\smallskip

We recall that, 
if $ M =S^2$ and  $ c > c_u(L)=c(L) $, 
 the  restriction of the  Lagrangian flow  to the   
energy level $ E_L^{-1}(c) $ is a  reparametrization  of a  geodesic flow  in the unit
tangent bundle for an appropriate   Finsler metric on $ S^2$  (cf. \cite[Corollary2]
{Contreras_Iturriaga_Paternain_Paternain:1998a}).  
As we  mentioned in the introduction,
in the setting of generic geodesic flows, the results of Franks \cite{Franks:92} and Bangert \cite{Bangert:93} assert that any Riemannian metric on $ S^2$ has infinitely many geometrically distinct closed geodesics.
Many results for geodesic flows
of a Riemannian metric  remain valid  for Finsler metrics but, in
contrast with the  Riemannian case,  there exist examples of bumpy
Finsler metrics on $ S^2$ with only two closed geodesics. These
examples were given by  Katok in \cite{Katok:1973} and were studied
geometrically by Ziller in \cite{Ziller:1983}. 
On the other hand,  a  result by   Rademacher shows that, if  the  Riemanniam manifold is  closed, simply connected manifold and its rational cohomology ring $
H^*(M,\mathbb{Q})$ is monogenic, then either there are   infinitely many closed geodesics,  or there are at 
least  one non-hyperbolic closed geodesic,  see  \cite[theorem 3.1(b)]{Rademacher:1989}.  
This result remains valid for bumpy Finsler metrics on $ S^2$ (cf. \cite[ Section 4 ]{Rademacher:1989}). So, we have the following proposition that is a particular case of the Rademacher's theorem.

\begin{prop}
\label{Rademacher}
Let $ F: TS^2 \to \R $ be a bumpy Finsler metric on $S^2$.  Suppose that there are only finitely many closed geodesics for $ F $ on $ S^2$. Then there is at least one non-hyperbolic closed geodesic.
\end{prop}

\smallskip

The following proposition proves the Theorem \ref{thm_surface} for high energy levels, i.e. $ c> c_u(L) $. 
 
\begin{prop}
\label{case hight energy}
Let $ L: TM \to \R$ be a Tonelli Lagrangian on a closed surface $M$ and let  $ c > c_u(L)$.  Then there is  a potential $u: M \to \R$, with the $C^2$-norm  arbitrarily small, such that $ h_{top}(L-u,c)>0.$ 
\end{prop}

\begin{proof}
Let  $ c > c_u(L)$. If $ u_0(x) \equiv 0 \notin \mbox{\rm int}_{C^2} \m H(L,c) $, then there exist a smooth potential  $u: M \to \R$,
 with the $C^2$-norm  arbitrarily small, such that  the corresponding flow $ \left.\phi^{L-u}_t \right|_{E_{L-u}^{-1}(c)} $ has a non-hyperbolic closed orbit, 
therefore the proposition follows from Proposition~\ref{energy with a non-hyperbolic orbit}.
Let us to assume that $ u_0(x) \equiv 0 \in \mbox{\rm int}_{C^2} \m H(L,c) $. 
 Recall that $\m G(1,L,c)=\m F^2(M)$. 
By Theorem~\ref{T_maim-u}, the set    $ \overline{ Per(L,u_0,c)}$ is  hyperbolic. 
We have two cases, $ M \not= S^2$ or $ M= S^2$.  If $ M \not= S^2$, it follows from  Lemma~\ref{prop infinite closed orbit for high energy}, that the hypotheses  of Corollary~\ref{cor_manyperiodic} holds, then $ h_{top}(L,c)>0$.    If $ M = S^2 $,
by the Kupka-Smale Theorem \cite{Oliveira:2008}
and Corollary~2 in \cite{Contreras_Iturriaga_Paternain_Paternain:1998a},  
we  can choose  $ u \in \mbox{\rm int}_{C^2} \m H(L,c) $ with the $C^2$-norm  arbitrarily small such that $c_u(L-u) < c $ (by    continuity of the critical values, cf. Lemma~2.2-1 in  \cite{Contreras_Iturriaga_Coloquio:1999a}), 
for which we  ensure  that  the flow $ \left.\phi^{L-u}_t \right|_{E_{L-u}^{-1}(c)} $  is a reparametrization of a geodesic flow for a bumpy   Finsler metric on $ M$.   
 Then   the Rademacher's result (Proposition~\ref{Rademacher})  implies that  $ \left.\phi^{L-u}_t \right|_{E_{L-u}^{-1}(c)} $ has and infinite number of closed orbits. 
Then the Corollary~\ref{cor_manyperiodic} imply that  $ h_{top}(L-u,c)>0$ .  
\end{proof}

\bigskip

For a small energy level case of a Tonelli Lagrangian on arbitrary surface $M$, we have the  following recent result proved by
  L. Asselle and M.Mazzucchelli   in \cite{Asselle_Mazzucchelli:2019}.

 \begin{thm}[\cite{Asselle_Mazzucchelli:2019}, Theorem 1.3]
 \label{Assele_Mazz_2}
Let $M$ be a closed surface and let $ L: TM \to \R $ be a Tonelli Lagrangian. Then there exists a set $ \m{E}(L) \subset ( e_0(L), c_u(L)) $ with full Lebesgue measure, such that, for every energy level $ \kappa \in  \m{E}(L)$ the Lagrangian flow of $ L$ has infinitely many closed orbits with energy $\kappa$.  
 \end{thm}

The following proposition completes the proof of Theorem~\ref{thm_surface}.

\begin{prop}\label{case low energy}
Let $M$ be a closed surface. Let   $L:TM\rightarrow \R$ be a  Tonelli Lagrangian and  $ c \in ( e_0(L), c_u(L)]$. Then there is  a potential $u: M \to \R$, with the $C^2$-norm arbitrarily small, such that $ h_{top}(L-u,c)>0.$
\end{prop}

\begin{proof}
Observe that if  $ u_k(x)=k $ denotes  a constant potential,  then the Euler-Lagrange equation \eqref{E-L} of $ L-u_k $ is the same equation of $ L$, 
\[E_{L-u_k}(x,v)=
 E_{L}(x,v) +k,\]
 and $ e_0(L-u_k)= e_0(L) +k $.
Let $  \m E(L) \subset (e_0,c_u(L))$ be given by Theorem~\ref{Assele_Mazz_2}, then  the subset $ \m{E}(L)$ has  full Lebesgue measure, and,  for  every $ \kappa \in  \m{E}(L)$ the  flow 
\[ \left. \phi_t^L \right|_{E^{-1}_L(\kappa)}: E^{-1}_L(\kappa) \rightarrow E^{-1}_L(\kappa)\]
 has infinitely many closed orbits. 
 Given  $ c \in  ( e_0(L), c_u(L)]$ we set $ \epsilon_0=(c-e_0(L))/{3}$.   
Since that the subset $  \m E(L)$ has full Lebesgue  measure as subset of  $ (e_0,c_u(L))$,
for all $ 0<\epsilon <\epsilon_0 $ we can take $ \kappa \in \m E(L)$,
such that $ |\kappa-c|< \epsilon$. 
So, using $k=\kappa-c$, we have that the corresponding Lagrangian flow of $ L-u_k$ are the same of $ L$ and  $  E^{-1}_{L-u_k}(c)= E^{-1}_L(\kappa)$.
 Therefore the flow $ \left. \phi_t^{L-u_k} \right|_{E^{-1}_{L-u_k}(c)} $ has  infinitely many closed orbits.  
 If $ u_k \in \mbox{\rm int}_{C^2}\m H(L,c)$,  by Corollary~\ref{cor_manyperiodic}, we have $ h_{top}(L-u_k,c)>0$.  If $ u_k \notin \mbox{\rm int}_{C^2} \m H(L,c) $, then there exists a smooth potential  $u: M \to \R$,
 with the $C^2$-norm  arbitrarily small, such that the corresponding flow $ \left.\phi^{L-(u_k+u)}_t \right|_{E_{L-(u_k+u)}^{-1}(c)} $ has a non-hyperbolic closed orbit, 
therefore the proposition follows from Proposition~\ref{energy with a non-hyperbolic orbit}.

\end{proof}

\section*{Acknowledgments}
We are grateful to C. Carballo and M. J. Dias Carneiro for the helpful conversations.
J. A. G. Miranda is very grateful to CIMAT for the hospitality and to CNPq-Brazil for the partial financial support.


\bibliographystyle{amsalpha}


\providecommand{\bysame}{\leavevmode\hbox to3em{\hrulefill}\thinspace}
\providecommand{\MR}{\relax\ifhmode\unskip\space\fi MR }
\providecommand{\MRhref}[2]{%
  \href{http://www.ams.org/mathscinet-getitem?mr=#1}{#2}
}
\providecommand{\href}[2]{#2}

\end{document}